\documentclass[twoside,a4paper,reqno,12pt]{amsart}
\usepackage{amsfonts, amsbsy, amsmath, amssymb, latexsym}
\usepackage{mathrsfs,array}
\usepackage[top=1 in,right=1 in,bottom=1 in,left=1 in]{geometry}
\usepackage[utf8]{inputenc}
\usepackage{hyperref}
\usepackage[english]{babel}
\usepackage{graphicx}
\usepackage{mathtools}
\usepackage{floatflt}
\usepackage{multicol}
\usepackage{lipsum}
\usepackage{diagbox}
\usepackage{wrapfig}
\usepackage{relsize}
\usepackage{booktabs}
\usepackage{pgfplots}
\usepackage{pgfplotstable}
\usepackage{enumerate}
\usepackage{blindtext}
\usepackage{amssymb}
\usepackage{amsthm}
\usepackage[mathscr]{euscript}
\usepackage{enumitem}
\usepackage{quiver}
\pgfplotsset{compat=1.18}
\usepackage{array,tabularx}

\usepackage{ytableau}

\usepackage[english]{babel}

\newtheorem{theorem}{Theorem}[section]

\newtheorem{proposition}[theorem]{Proposition}
\newtheorem{corollary}[theorem]{Corollary}

\theoremstyle{definition}
\newtheorem{definition}[theorem]{Definition}
\newtheorem{example}[theorem]{Example}

\theoremstyle{remark}
\newtheorem{remark}[theorem]{Remark}

\numberwithin{equation}{section}

\newcommand{\Z}{\mathbb{Z}}
\newcommand{\R}{\mathbb{R}}
\newcommand{\N}{\mathbb{N}}

\newcommand{\Q}{\mathbb{Q}}
\newcommand{\inv}{^{-1}}
\newcommand{\sigper}{S\mathfrak{S}}

\newcommand{\F}{\mathcal{F}}

\newcommand{\shuffle}{\sqcup\mathchoice{\mkern-7mu}{\mkern-7mu}{\mkern-3.2mu}{\mkern-3.8mu}\sqcup}

\def\shuff#1#2{\mathbin{
      \hbox{\vbox{\hbox{\vrule \hskip#2 \vrule height#1 width 0pt}\hrule}\vbox{\hbox{\vrule \hskip#2 \vrule height#1 width 0pt\vrule }\hrule}}}}
\def\shuffl{{\mathchoice{\shuff{5pt}{3.5pt}}{\shuff{5pt}{3.5pt}}{\shuff{3pt}{2.6pt}}{\shuff{3pt}{2.6pt}}}}
\def\shuffle{\, \shuffl \,}

\def\qshuffl{{\mathchoice{\shuff{5pt}{3.5pt}\hspace{-2.9mm}-}{\shuff{5pt}{3.5pt}\hspace{-2.9mm}-}
{\shuff{3pt}{2.6pt}\hspace{-2.2mm}-}{\shuff{3pt}{2.6pt}\hspace{-2.2mm}-}}}
\def\qshuffle{\,\qshuffl\,}

\newcommand{\bicomposition}{\big(\begin{smallmatrix}
a_1 & a_2 & \cdots & a_r\\
b_1 & b_2 & \cdots & b_r
\end{smallmatrix}\big)}
\newcommand{\dgraph}{\overrightarrow{G}}
\newcommand{\sdgraph}{\overrightarrow{\Sigma}}

\title{Chromatic quasisymmetric functions for signed graphs}

\author{Jean-Christophe Aval}
\address{LaBRI, Université de Bordeaux \\
Domaine Universitaire, 351, Cours de la Libération, 33405 Talence \\
France}
\email{aval@labri.fr}

\author{Raquel Melgar}
\address{LaBRI, Université de Bordeaux \\
Domaine Universitaire, 351, Cours de la Libération, 33405 Talence \\
France}
\email{raquel.melgar@labri.fr}

\thanks{Both authors were supported by the ANR (Combiné Project ANR-19-CE48-0011).}

\begin{document}

\begin{abstract}
In 1995, Stanley introduced the chromatic symmetric function of a graph, which specializes to its chromatic polynomial, and which has been the focus of intense research. In 2017, Shareshian, Wachs, and Ellzey defined a refinement of this function for a directed graph, that appears to be in $QSym$, the algebra of quasisymmetric functions, which is of great interest in algebraic combinatorics. Our goal is to extend this work to signed graphs, taking into account the perspective of the hyperplane arrangement associated with a signed graph, developed by Zaslavsky. We introduce the signed chromatic quasisymmetric invariant, and obtain structural properties. As a consequence, we define and study $SQSym$, the algebra of signed quasisymmetric functions.
\end{abstract}

\maketitle
\tableofcontents

\section{Introduction}
Defined in 1995 by Stanley \cite{S95}, the chromatic symmetric function plays a major role among graph invariants. Let $G$ be a (simple, non oriented) graph with vertex set $V$ and edge set $E$. 
A proper coloring of $G$ is a map $\kappa:V\rightarrow \N$ from $V$ to the positive integers $\N=\{1,2,\dots\}$, with the condition $\kappa(u)\neq \kappa(v)$ if $\{u,v\}\in E$: adjacent vertices are assigned different colors. We consider the infinite set of commuting variables $x=(x_1,x_2,\dots)$ and associate to a coloring $\kappa$ the monomial 
$$x^\kappa:=x_{\kappa(v_1)}x_{\kappa(v_2)}\cdots x_{\kappa(v_d)}=\prod_{i\in \N}x_i^{\#\kappa\inv(i)}.$$ 
Then, \textit{the chromatic symmetric function of $G$} is defined as \begin{equation}
X_G(x):= \sum_{\kappa \text{ proper}}x^\kappa.
\end{equation}
It is clear that $X_G\in Sym$, where $Sym$ is the algebra of symmetric functions, since a permutation of the variables is just a renaming of the colors. A remarkable feature of this function is that it specializes in the chromatic polynomial $\chi_G(k)$, i.e. the function that sends each integer $k\geq 0$ to the number of proper colorings of $G$ on $k$ colors. Indeed,
$$\chi_G(k):=\#\{ \text{proper colorings }\kappa\text{ of }G\text{ on }k \text{ colors}\} = X_G(1,1,\dots,1,0,\dots),$$ 
where in the right hand side of the equality we have $k$ ones.

In the same paper, Stanley proposed several conjectures that remain unsolved to this day. This symmetric function has generated extensive research since its introduction, see for instance \cite{GS01}, \cite{OS14} and \cite{MMW08}.

Following \cite{SW16},  Ellzey \cite{E17} introduced a refinement of Stanley's chromatic symmetric function for directed  graphs. Given a simple directed  graph $\overrightarrow{G} = (V,E)$, with $E$ the set of directed edges, the \textit{chromatic quasisymmetric function of $\dgraph$} is
\begin{equation}
\label{equation: Ellzey}
X_{\dgraph}(x;t) := \sum_{\kappa\text{ proper}} t^{asc(\kappa)}x^\kappa
\end{equation}
where $asc(\kappa) = \#\{(u,v)\in E:~\kappa(u)<\kappa(v)\}$. It is straightforward that, by evaluating this function at $t=1$, we recover Stanley's invariant. Also, it is easy to check that $X_{\dgraph}(x;t)$ lies in $QSym[t]$. The importance of the algebra $QSym$ of quasisymmetric functions in algebraic combinatorics is well-known (see for instance the survey \cite{M19}). This invariant is a generalization of the quasisymmetric function introduced in \cite{SW16} for labeled graphs, with the difference that $X_{\dgraph}(x;t)$ also performs with directed  graphs that may contain directed   cycles.

The notion of \textit{signed graphs} was introduced by Harary \cite{H53}. These are graphs in which a sign ($+$ or $-$) is assigned to every edge; we shall talk about positive or negative edges respectively. Zaslavsky \cite{Z82} defined the notion of \textit{proper coloring for signed graphs}. Given a signed graph $\Sigma$ with vertex set $V$, a proper coloring of $\Sigma$ is a map $\kappa:V\rightarrow \Z$ verifying that if $e =\{u,v\}$ is an edge of $\Sigma$ with sign $\epsilon$, then $\kappa(u)\neq \epsilon \kappa(v)$. In other words, a proper signed coloring assigns to the vertices colors from $\Z$ in such a way that vertices connected by a positive edge cannot have the same color, and vertices connected by a negative edge cannot have opposite colors. We call a coloring of $\Sigma$ that does not use the color 0 a \textit{zero-free signed coloring}.

A generalization of Stanley's chromatic symmetric function for the case of signed graph coloring was first introduced in \cite{W97}, then studied in \cite{KT21} and \cite{E16} (this last in the zero-free case). We consider in this case the set of commuting variables $\textbf{x} := (\dots,x_{-1},x_0,x_1,\dots)$. Given a signed graph on $d$ vertices $\Sigma$, its \textit{chromatic signed symmetric function} is
\begin{equation}
X_{\Sigma}(\textbf{x}) := \sum_{\kappa\text{ proper}}\textbf{x}^\kappa = \sum_{\kappa\text{ proper}}x_{\kappa(v_1)}x_{\kappa(v_2)}\cdots x_{\kappa(v_d)}.
\end{equation} 
Whereas the original invariant $X_G$ is a symmetric function, 
the invariant $X_\Sigma$ has a signed symmetry property.
Let $S\mathfrak{S}$ be the signed symmetric group, that is the set of permutations of $\Z$ (with finite support) verifying $\pi(i)= \pi(-i)$ for all $i\in \Z$.
Then $X_\Sigma$ is invariant under the permutation of its variables by any element of $S\mathfrak{S}$.
We denote by $SSym$ the set of formal power series of bounded degree in the variables $\textbf{x} := (\dots,x_{-1},x_0,x_1,\dots)$ with coefficients in $\Q$ that are invariant under the action of $S\mathfrak{S}$ on its variables. This set is in fact an algebra which is studied in \cite{KT21}. 

In \cite{Z91} Zaslavsky develops a theory for the orientations of signed graphs and proves that there exists a bijection between the acyclic orientations of a signed graph on $d$ vertices $\Sigma$ and the connected components of the complement of an hyperplane arrangement $\mathscr{H}_\Sigma \subset \mathbb{R}^{d}$. This remarkable result generalizes the well-known analog result of Green \cite{G77} for the classical case.

The aim of this paper is, taking this interpretation into account, to define an invariant $X_{\sdgraph}(\textbf{x};t)$ for directed signed graphs $\sdgraph$ that refines $X_\Sigma(\textbf{x})$, in the same way as $X_{\dgraph}(x;t)$ refined $X_{G}(x)$. In other words, we want to fill the gap in the table.

\begin{table}[h]
\label{table: fill-the-gap}
\begin{tabular}{c|c|c|}
\cline{2-3}
                                     & undirected  & directed  \\ \hline
\multicolumn{1}{|l|}{simple graph} &   $X_G(x)\in Sym$  &  $X_{\dgraph}(x;t)\in QSym[t]$   \\ \hline
\multicolumn{1}{|l|}{signed graph}              &  $X_{\Sigma}(\textbf{x})\in SSym$    &     ? \\ \hline
\end{tabular}
\end{table}

In Section \ref{sec:sec2} we define this new invariant $X_{\sdgraph}(\textbf{x};t)$ and explain its relation with the hyperplane arrangement $\mathscr{H}_\Sigma$. In Section \ref{sec:sec3} we are interested in the space where this invariant lives and we call it $SQSym$ as it is the analogue of $QSym$ in the signed case. 
We prove that $SQSym$ is an algebra. We propose a fundamental family for this algebra and we describe its relation with the $(P,\omega)$-partition enumerators of signed posets. In Section \ref{sec:sec4} we give an expansion of $X_{\sdgraph}(\textbf{x};t)$ in the fundamental family. And in Section \ref{sec:sec5} we investigate the symmetry question of $X_{\sdgraph}(\textbf{x};t)$. In other words, we study for which directed  signed graphs $\sdgraph$, one has $X_{\sdgraph}(\textbf{x};t)\in SSym[t]$.

\section{The definition of the invariant}
\label{sec:sec2}
	\subsection{Preliminaries on signed graphs}
A \textit{signed graph} $\Sigma$ is a triple $\Sigma=(V,E,\sigma)$, with $V$ a vertex set, $E$ an edge set, and a \textit{sign function} (or \textit{signature}) $\sigma: E\rightarrow \{+,-\}$ that assigns each edge either a $+$ or a $-$ sign. 
We impose the condition that the pairs $(e,\sigma(e))$ are all distinct.
We call $G=(V,E)$ the \textit{underlying unsigned graph of $\Sigma$}; it may have double edges (corresponding to two edges with different signs in $\Sigma$) and it may also have loops.

Given a path $P=(e_1,\dots,e_t)\in E^t$ the \textit{sign of $P$} is the product $\sigma(P):=\sigma(e_1)\cdots \sigma(e_t)$ and we say that $P$ is \textit{balanced} if $\sigma(P)=+$. We say a signed graph $\Sigma$ is \textit{balanced} if given any cycle (closed path with no repeated vertices) $C$ in $\Sigma$, $\sigma(C)=+$. Otherwise we say $\Sigma$ is \textit{unbalanced}.

A \textit{tight handcuff} is the union of two unbalanced cycles sharing exactly one vertex, a \textit{loose handcuff} is the union of two unbalanced cycles sharing no vertices and a path that joins them. A \textit{frame circuit} of a signed graph $\Sigma = (V,E,\sigma)$ is a subgraph of $\Sigma$ that is either a balanced cycle, a tight handcuff or a loose handcuff.

\begin{figure}[h]
    \centering
    \includegraphics[width=0.25\linewidth]{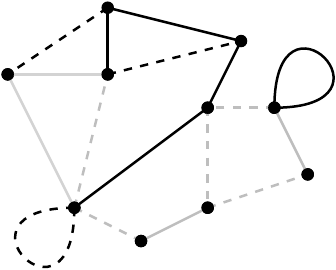}
    \caption{A signed graph, where the negative edges have been drawn with dashed lines, with one of its frame circuits in gray.}
    \label{fig: frame-circuit}
\end{figure}

A \textit{switching function} for $\Sigma = (V,E,\sigma)$ is a map $\nu:V\rightarrow \{+,-\}$. Such a function acts on signatures and thus on signed graphs as follows. A signature $\sigma$ may be switched by $\nu$ to get ${\sigma^\nu} : E  \rightarrow \{+,-\}$ defined as $\sigma(e):=\nu(u)\sigma(e)\nu(v)$ where $u,v\in V$ are the endpoints of $e\in E$. And $\Sigma^\nu:=(V,E,\sigma^\nu)$ may be called $\Sigma$ \textit{switched by} $\nu$. We say that two signed graphs $\Sigma=(V,E,\sigma)$ and $\Sigma'=(V,E,\sigma')$ are \textit{switching equivalent} if there exists a switching function $\nu$ such that $\sigma'=\sigma^\nu$. 

\begin{remark}Switching does not alter balance; given a signed graph $\Sigma = (V,E,\sigma)$ and a switching function $\nu$ then $\Sigma$ is balanced if and only if $\Sigma^\nu$ is balanced. To see this, take a cycle $C=(e_1,e_2,\dots,e_t)$ on $G=(V,E)$, it is easy to check that $\sigma(C)=\sigma^\nu(C)$.

It can be shown \cite{Z82+} that a signed graph is balanced if and only if it is switching equivalent to $\Sigma^+$, the signed graph obtained by letting all edges in $\Sigma$ be positive.
\end{remark}

\begin{figure}
    \centering
    \includegraphics[width=0.65\linewidth]{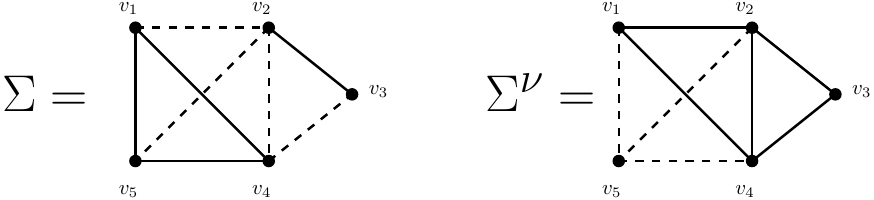}
    \caption{A signed graph $\Sigma$ and $\Sigma^\nu$ where $\nu:V\rightarrow \{+,-\}$ sends every vertex to $+$ except $v_1$ and $v_4$.}
    \label{fig: swiched-sigma}
\end{figure}

\subsection{Directed signed graphs}
A \textit{bidirection} of a graph $G=(V,E)$  is any map $\tau:I(G)\rightarrow\{+,-\}$, where the \textit{incidence set} $I(G)$ is the set of pairs $(u,e)$ where $u\in V$ is an endpoint of $e\in E$. The interpretation of $\tau$ is the following: if $\tau(u,e)=+$ then the indicence $(u,e)$ points into the vertex $u$ and if $\tau(u,e)=-$, the incidence points away from $u$. Given a signed graph $\Sigma=(G,\sigma)$ we call an \textit{orientation of $\Sigma$} a bidirection $\tau$ of $G=(V,E)$ verifying 
\begin{equation}
\label{equation: condicion-orientation}
\sigma(e)=-\tau(u,e)\tau(v,e)
\end{equation}
on every edge $e = \{u,v\}\in E$. A \textit{directed signed graph} $\sdgraph$ is a pair $(\Sigma, \tau)$ where $\Sigma = (V,E,\sigma)$ is a signed graph and $\tau$ satisfies \eqref{equation: condicion-orientation}. This means that in a directed  signed graph $\sdgraph$, positive edges are directed in the usual way but negative edges are directed with the arrows pointing inwards (introverted edges) or outwards (extroverted edges).

\begin{figure}
    \centering
    \includegraphics[width=0.8\linewidth]{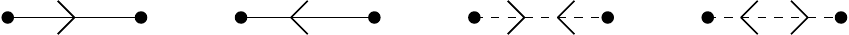}
    \caption{Possible orientations of an edge}
\end{figure}

By a slight abuse, we will allow to have two edges of the same sign between two vertices if they are directed in opposite ways, the same applies to loops. We call $\Sigma$ the \textit{underlying undirected  graph of} $\sdgraph$, in which we eliminate edges of the same sign that repeat. We may also write $\sdgraph=(V,E,\sigma,\tau)$.

Given a directed  signed graph $\sdgraph=(V,E,\sigma, \tau)$ and a switching function $\nu$, we define the \textit{switched directed signed graph} as $\sdgraph^\nu:=(V,E,\sigma^\nu, \tau^\nu)$ where $\tau^\nu(u,e):=\nu(u)\tau(u,e)$.

\begin{figure}[h]
    \centering
    \includegraphics[width=0.4\linewidth]{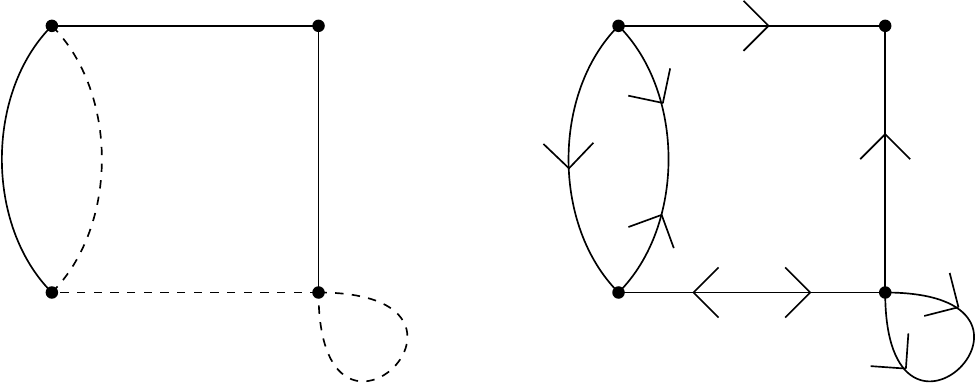}
    \caption{A signed graph and one of its possible orientations.}
\end{figure}
\begin{definition}
Given a signed graph $\Sigma = (V,E,\sigma)$, a \textit{coloring of} $\Sigma$ is any map $\kappa: V\rightarrow \Z$. We say that $\kappa$ is a \textit{proper coloring} if
\begin{equation}
\kappa(u)\neq \sigma(e)\kappa(v)\text{ for all } e\in E
\end{equation} where $u,v\in V$ are the endpoints of $e$. Given a switching function $\nu$ we note $\kappa^\nu$ the coloring $\kappa^\nu(u):=\nu
(u)\kappa(u)$ for all $u\in V$.
\end{definition}

\begin{remark} The signed graphs that we consider in this paper are graphs without half edges; these correspond to what Zaslavsky calls \textit{ordinary graphs} in his work. It is sufficient to consider only these, since in our context half edges play exactly the same role as negative loops. We will also only consider negative loops, and not positive ones, since by definition the latter do not allow proper colorings. We do not treat zero-free colorings as a separate case, since it suffices to add a negative loop at each vertex to restrict to zero-free colorings. 
\end{remark}

\begin{definition}
Given a signed graph $\Sigma=(V,E,\sigma)$, we say that a coloring $\kappa$ and an orientation $\tau$ are \textit{compatible} on an edge $e\in E$ with endpoints $u$ and $v$ if 
\begin{equation}
    \tau(u,e)\kappa(u)+\tau(v,e)\kappa(v)\leq 0.
\end{equation}
We say that $\kappa$ and $\tau$ are \textit{compatible} if they are compatible on every edge $e\in E$.
\end{definition}

Note that given a switching function $\nu:V\rightarrow \{+,-\}$, $\kappa$ is a proper coloring of $\Sigma = (V,E,\sigma)$ if and only if $\kappa^\nu$ is a proper coloring of $\Sigma^\nu$ and given an orientation $\tau$, $\kappa$ and $\tau$ are compatible on $e\in E$ if and only if $\kappa^\nu$ and $\tau^\nu$ are compatible on $e$.

\begin{remark}
    The convention of compatibility between colorings and orientations follows Zaslavksy \cite{Z82}. If $(u,v)$ is a positive edge directed from $u$ to $v$, a coloring $\kappa$ is compatible with $(u,v)$ if $\kappa(u)>\kappa(v)$. In \cite{SW16,E17}, the opposite convention is used.
\end{remark}

\begin{definition}
    Given a signed graph $\Sigma$ with an orientation $\tau$, we say that $\tau$ is \textit{acyclic} if every frame circuit in $\Sigma$ has a source or a sink, i.e. every frame circuit in $\Sigma$ has a vertex $v$, either with all arrows pointing inwards or with all arrows pointing outwards.
\end{definition}
This definition means that in an acyclic orientation, in every closed path we can find a source or a sink. 

\subsection{Symmetric graphs}
There is a natural correspondence between signed graphs on $d$ vertices and symmetric graphs on $2d$ vertices. The latter are also called \textit{double covering graphs} in \cite{Z82+}. This point of view not only sheds some light on the constructions, but it is crucial to define specific objects, as the poset associated to a directed signed graph (Section \ref{sec:sec4}).

\begin{definition}
    A \textit{symmetric graph} is a graph $G= (V, E)$ with vertex set $$V=\{u_1,u_{-1},u_2,u_{-2},\dots,u_d,u_{-d}\}$$ such that if $\{u_i,u_j\}\in E$ then $\{u_{-i},u_{-j}\}\in E$.
\end{definition}

The correspondence is as follows.
\begin{table}[h]
\begin{tabular}{|c|c|}
\hline
Signed graph  $\Sigma$       & Symmetric graph $G_\Sigma$                              \\ \hline
Vertex $v_i$                 & Pair of vertices $u_i$, $u_{-i}$                        \\ \hline
Positive edge  $\{v_i,v_j\}$ & Edges $\{u_i, u_j\}$ and $\{u_{-i},u_{-j}\}$            \\ \hline
Negative edge  $\{v_i,v_j\}$ & Edges $\{u_i, u_{-j}\}$ and $\{u_{-i},u_{j}\}$ \\ \hline
Positive loop in $v_i$       & Loops in $u_i$ and $u_{-i}$                             \\ \hline
Negative loop in $v_i$       & Edge $\{u_i,u_{-i}\}$                               \\ \hline
\end{tabular}
\end{table}

Colorings and orientations of signed graphs also have their analogues on symmetric graphs.

\begin{definition}
    A \textit{proper coloring of a symmetric graph} $G = (V,E)$ with $V=\{u_1,u_{-1},\dots,u_d,u_{-d}\}$ is a map $\kappa: V \rightarrow \Z$ verifying 
    \begin{itemize}
        \item $\kappa(u_i)=-\kappa(-u_i)$ for every $i\in\{1,\dots,d\}$
        \item $\kappa(u_i)\neq \kappa(u_j)$ for every edge $\{u_i,u_j\}\in E$  \end{itemize}
\end{definition}

\begin{definition}
    A \textit{symmetric orientation of a symmetric graph } $G = (V,E)$ is an orientation of $G$ such that if $\{u_i,u_j\}\in E$ is directed from $u_i$ to $u_j$, then $\{u_{-i},u_{-j}\}$ is directed from $u_{-j}$ to $u_{-i}$. We say that a symmetric orientation is \textit{acyclic} if it is acyclic in the usual sense.
\end{definition}

From a coloring $\kappa$ of $\Sigma$ we obtain a symmetric coloring $\tilde{\kappa}$ of $G_\Sigma$ by letting $\tilde{\kappa}(u_i)=\kappa(v_i)$ and $\tilde{\kappa}(u_{-i})=-\kappa(v_i)$ for $i=1,2,\dots,d$. And from an orientation  $\tau$ of $\Sigma$ we can obtain a symmetric orientation of $G_\Sigma$ in the following way, where $(u_i,u_j)$ represents a directed  edge from $u_i$ to $u_j$.
\begin{table}[h]
\begin{tabular}{|l|l|}
\hline
Positive edge from $v_i$ to $v_j$        & $(v_i,v_j)$ and $(v_{-j},v_{-i})$     \\ \hline
Introverted edge between $v_i$ and $v_j$ & $(v_i,v_{-j})$ and $(v_j,v_{-i})$     \\ \hline
Extroverted edge between $v_i$ and $v_j$ & $(v_{-i},v_{j})$ and $(v_{-j},v_{i})$ \\ \hline
\end{tabular}
\end{table}

\begin{figure}
    \centering
    \includegraphics[width=0.5\linewidth]{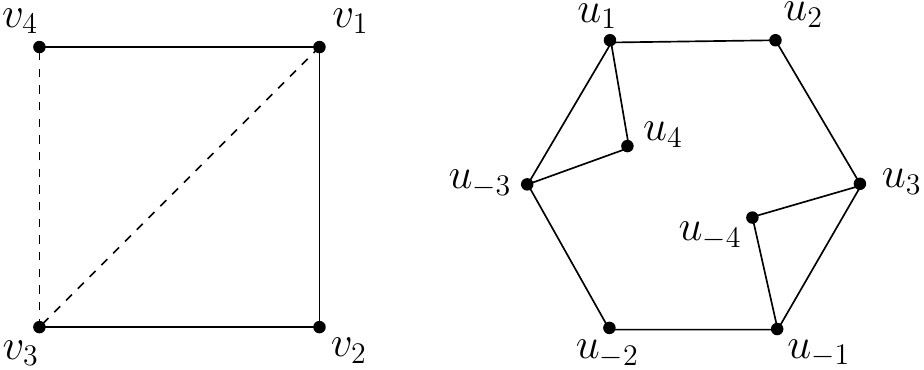}
    \caption{A signed graph and its corresponding symmetric graph}
    \label{fig: correspondence-symmetric}
\end{figure}

These correspondences send proper colorings of $\Sigma$ to proper symmetric colorings of $G_\Sigma$ and acyclic orientations of $\Sigma$ to acyclic symmetric orientations of $G_\Sigma$ respectively.
 
	\subsection{Hyperplane arrangements}
The aim of this subsection is to interpret the quasisymmetric invariant of \eqref{equation: Ellzey} defined in \cite{E17}, in terms of arrangements of hyperplanes. This point of view is a guideline for the definition of our invariant.

Given a simple graph $G = (V,E)$ with $V = \{v_1,\dots,v_d\}$ we can associate to it an hyperplane arrangement $\mathscr{H}_G$ in $\R ^d$ $$\mathscr{H}_G:=\bigcup_{\{v_i,v_j\}\in E}\{(\xi_1,\dots,\xi_d):~\xi_i=\xi_j\}.$$ Then we can associate to every integer point $\alpha = (\alpha_1,\dots,\alpha_d)\in\N^d\backslash \mathscr{H}_G$ the coloring $\kappa_\alpha: V \rightarrow \N$ that sends $v_i$ to $\alpha_i$ for each $i= 1,\dots,d$, which is proper since $\alpha\notin \mathscr{H}_G$. We call \textit{chambers of $\mathscr{H}_G$} the connected components of $\R^d\backslash \mathscr{H}_G$. We recall the following classical result.
\begin{proposition}\label{prop:green}
\cite{G77} Given a simple graph $G$, there exists a bijection between the chambers of $\mathscr{H}_G$ and the acyclic orientations of $G$.
\end{proposition}

The idea is that given $\xi= (\xi_1,\dots,\xi_d)\in \R^d\backslash \mathscr{H}_G$ we can orient an edge $\{v_i,v_j\}\in E$ as $(v_i,v_j)$  if $\xi_i>\xi_j$. It is not difficult to see that if we orient every edge following this rule we will get an acyclic orientation and that we will obtain the same orientation for every $\xi$ in the same chamber.
\begin{proposition}
Given a directed  graph $\dgraph = (V,E)$ and a chamber $C$ of $\mathscr{H}_G$,  $asc(\kappa_\alpha)=\#\{(v_i,v_j)\in E:~ \kappa_\alpha(v_i)<\kappa_\alpha(v_j)\}$ is constant for every $\alpha\in \N^d\cap C$.
\end{proposition} 
\begin{proof}
It is clear since for every edge $(v_i,v_j)\in E$, as $\{\xi_i=\xi_j\}\in \mathscr{H}_G$, the condition $\xi_i<\xi_j$ either holds or does not hold for every $\xi\in C$. 
\end{proof}

 Then the quasisymmetric invariant of \eqref{equation: Ellzey} admits the following interpretation: 
 \[X_{\dgraph}(x_1,x_2,\dots;t)=X_{\dgraph}(x;t)=\sum\limits_{C}t^{asc(C)}\sum\limits_{\alpha\in C\cap \N^d}x_{\alpha_1}\cdots x_{\alpha_d}, \] 
 where the first sum runs over the chambers of $\mathscr{H}_G$ and $asc(C)$ is $asc(\kappa_\alpha)$ for any $\alpha\in C \cap \N^d$. Note that the edges counted by $asc(\kappa)$ are the edges in which $\kappa$ is not compatible with the orientation in $\dgraph$.
 
 We shall use this geometric point of view in the signed case. Given a signed graph $\Sigma = (V, E, \sigma)$ with $V = \{ v_1,\dots,v_d\}$, we can associate to it the hyperplane arrangement $\mathscr{H}_\Sigma$ in $\R^d$ defined by 
\[\mathscr{H}_\Sigma:= \big[\bigcup_{\{v_i,v_j\}\in \sigma\inv (+)}\{\xi_i= \xi_j\}\big]\cup \big[\bigcup_{\{v_i,v_j\}\in \sigma\inv (-)}\{\xi_i= -\xi_j\}\big] \cup \big[ \bigcup_{\{v_i,v_i\}\in \sigma\inv (-)}\{\xi_i= 0\}\big].\]

That is, for every positive edge $\{v_i,v_j\}$ we consider the hyperplane $\xi_i = \xi_j$, for every negative edge $\{v_i,v_j\}$ we consider the hyperplane $\xi_i =  - \xi_j$, and for every negative loop in $v_i$ we consider the hyperplane $\xi_i=0$. 

Here again, we have a bijection between $\Z^d\backslash \mathscr{H}_\Sigma$ and the proper colorings of $\Sigma$ and we also have the following result, analogous to Proposition \ref{prop:green}.
\begin{theorem}
\cite{Z91} Given a signed graph $\Sigma$, there is a bijection between the chambers  of $\mathscr{H}_\Sigma$ and the acyclic orientations of $\Sigma$. 
\end{theorem}

\begin{example}\label{ex:hyper}
Let $\Sigma$ be the signed graph consisting of two vertices, $v_1$ and $v_2$, connected by both a positive and a negative edge, with an additional negative loop at $v_1$. Figure \ref{figure: double example chambers} illustrates the acyclic orientations of $\Sigma$ and of $G_\Sigma$, the symmetric graph associated to $\Sigma$, each placed in the corresponding chamber of $\mathscr{H}_\Sigma \subset \R^2$.
    \begin{figure}[h]
        \centering
        \includegraphics[width=0.85\linewidth]{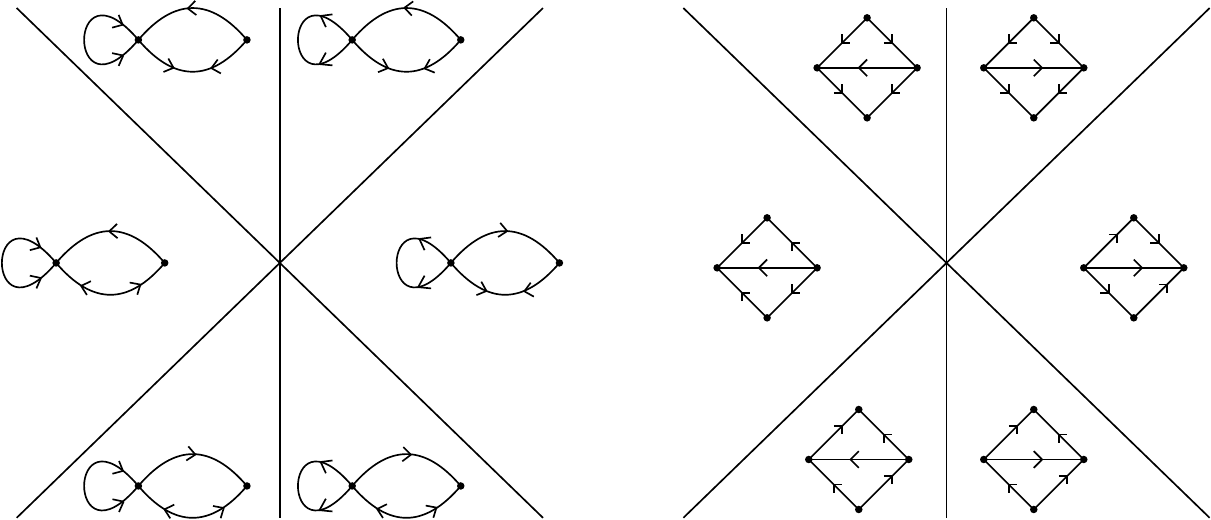}
        \caption{Hyperplane arrangement and acyclic orientations of Example~\ref{ex:hyper}
        \label{figure: double example chambers}}
    \end{figure}
\end{example}

\subsection{Main definition and basic results}
                           
\begin{definition}
Given a directed signed graph $\sdgraph = (V,E,\sigma, \tau)$, we define its \textit{chromatic signed quasisymmetric function} as
\begin{equation}
X_{\sdgraph}(\dots,x_{-1},x_0,x_1,\dots;t)=X_{\sdgraph}(\textbf{x};t) : = \sum_{\kappa \text{ proper}}t^{asc(\kappa)}\textbf{x}^\kappa 
\end{equation}
where $asc(\kappa):=\#\{e=\{v_i,v_j\}\in E:~ \tau(v_i,e)\kappa(v_i)+ \tau(v_j,e)\kappa(v_j)> 0\}$.
\end{definition}

As required, this definition is compatible
with the interpretation in terms of hyperplane arrangements, as:
\begin{equation}
\label{equation: suma-signed-chambers}
X_{\sdgraph}(\textbf{x};t)= \sum_{C}t^{asc(C)}\sum\limits_{\alpha\in C\cap\Z^d}x_{\alpha_1}\cdots x_{\alpha_d}
\end{equation}
where the first sum runs over the chambers of $\mathscr{H}_\Sigma$.

It is clear that $X_{\sdgraph}(\textbf{x};1)=X_{\Sigma}(\textbf{x})\in SSym$ is the invariant defined in \cite{KT21}, and that $X_{\sdgraph}(\textbf{x};1)|_{x_0=0}$ is the invariant defined in \cite{E16}.

\begin{example}\label{ex:the-example}
Consider the directed signed graph $\sdgraph$ in Figure \ref{figure: example-deux-sommets}. Then 
\begin{align*}
X_{\sdgraph}(\textbf{x};t)  = & t^3[\sum_{i<j}x_ix_j] + t^2[2\sum
_{i<j}x_{-i}x_{j}+\sum_{i<j}x_ix_j+x_0\sum_{i}x_i]\\
& + t[2\sum_{i<j}{x_ix_{-j}}+\sum_{i<j}x_{-i}x_{-j}+x_0\sum_{i}x_{-i}] + t^0[\sum_{i<j}x_{-i}x_{-j}]
\end{align*}
where the indices run through the set of positive integers $\N$.

\begin{figure}[h]
\includegraphics[width=80mm]{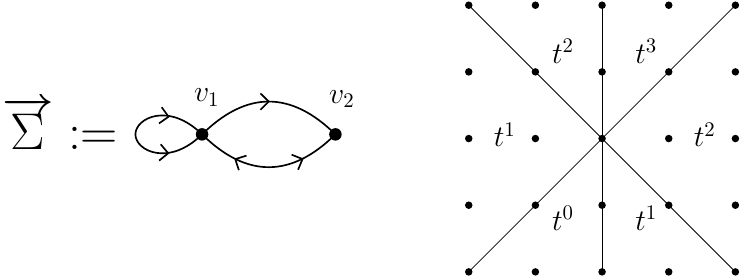}
\caption{A directed  signed graph and its corresponding arrangement in $\R^2$.}
\label{figure: example-deux-sommets}
\end{figure}
\end{example}

We start with an easy statement.
\begin{proposition}
    If $\sdgraph_1+\sdgraph_2$ denotes the disjoint union of two directed signed graph $\sdgraph_1$ and $\sdgraph_2$, then one has:
    $$ X_{\sdgraph_1+\sdgraph_2}(\textbf{\em{x}};t) = X_{\sdgraph_1}(\textbf{\em{x}},t)X_{\sdgraph_2}(\textbf{\em{x}},t).$$
\end{proposition}

\section{The algebra SQSym}
\label{sec:sec3}

	\subsection{Definition, monomial basis and dimension}
Let $\N_0=\N\cup\{0\}=\{0,1,2,\dots\}$.
Let $a,b \in \N_0^r$, then the pair $\lambda=(a,b)$ is called a \textit{bicomposition} if for every $i\in\{1,\dots,r\}$, $(a_i,b_i)\neq (0,0)$. We will use the $2 \times r$ matrix notation for bicompositions $\lambda=\bicomposition$. The \textit{degree} of $\lambda$ is the sum of all its entries, this is: $|\lambda|:=\sum_{i=1}^r (a_i+b_i)$.

\begin{definition}
Given a formal power series in the set of variables $\textbf{x}=(\dots,x_{-1},x_0,x_1,\dots)$ with rational coefficients, $f\in \Q[[\dots,x_{-1},x_0,x_1,\dots]]$, we say that $f$ is a \textit{signed quasisymmetric function} if:
\begin{itemize}
\item $f$ has bounded degree
\item given two sequences of natural numbers $i_1<i_2<\cdots<i_r$, $j_1<j_2<\cdots <j_r$ and a pair $(k, \lambda)$ where $k\in \N_0$ and $\lambda=(a,b)$ is a bicomposition we have that
$$[x_0^k x_{i_1}^{a_1}x_{-i_1}^{b_1}\cdots x_{i_r}^{a_r}x_{-i_r}^{b_r}]f=[x_0^k x_{j_1}^{a_1}x_{-j_1}^{b_1}\cdots x_{j_r}^{a_r}x_{-j_r}^{b_r}]f.$$ 
\end{itemize}
We denote by  $SQSym$ the set of signed quasisymmetric functions and by $SQSym ^d$ the set of all homogeneous signed quasisymmetric functions of degree $d$, where the degree of a monomial $x_0^k x_{i_1}^{a_1}x_{-i_1}^{b_1}\cdots x_{i_r}^{a_r}x_{-i_r}^{b_r}$ is $k+\sum_{i=1}^r(a_i+b_i)$.
\end{definition}
The set $SQSym$ also admits a characterization as invariant under an action (quasisymmetrizing action) of the symmetric group, as done by Hivert in the classical case \cite{Hi00}. For this, given $k\in\N_0$, a bicomposition $\lambda=\bicomposition$ and a set of natural numbers $I=\{i_1,i_2,\dots,i_r\}$ that has been listed in increasing order $i_1<i_2<\cdots <i_r$, we will note $\textbf{x}_I^{(k, \lambda)}:=x_0^k x_{i_1}^{a_1}x_{-i_1}^{b_1}\cdots x_{i_r}^{a_r}x_{-i_r}^{b_r}$. Then a permutation $\sigma$ acts on the monomials
in $\Q[[\dots,x_{-1},x_0,x_1,\dots]]$ as follows 
\begin{equation}
\label{eq:hivert}    
\sigma \ast \textbf{x}_I^{(k, \lambda)} = \textbf{x}_{\sigma I}^{(k, \lambda)}.
\end{equation}
One must be careful to list the elements of $\sigma I$ in increasing order. For instance if $\sigma= 3 1 5 2 4$ $$\sigma \ast x_0^3x_1x_3x_{-3}^2x_{-4}^3 = \sigma \ast \textbf{x}_{\{1,3,4\}}^{3, \big(\begin{smallmatrix}
1 & 1  & 0\\
0 & 2  & 3
\end{smallmatrix}\big)}=\textbf{x}_{\{2,3,5\}}^{3, \big(\begin{smallmatrix}
1 & 1  & 0\\
0 & 2  & 3
\end{smallmatrix}\big)}=x_0^3x_2x_3x_{-3}^2x_{-5}^3.$$

Then we have:
$$SQSym=\{f\in \Q[[\dots,x_{-1},x_0,x_1,\dots]] : f \text{ has bounded degree and } \sigma \ast f = f ~ \forall \sigma \in \mathfrak{S}\}.$$

\begin{remark}
This idea of extending Hivert's quasisymmetrizing action to a double set of variables was used in \cite{ABB04} to study \textit{diagonally quasisymmetric functions}. 
\end{remark}

As $QSym$, the vector space $SQSym$ can also be characterized by a Bergeron-Sottile operator \cite{BS98}. This property allows to quickly prove that it is an algebra as is done in \cite{NST24}.

    \begin{definition}
        Let $f\in \Q[[\dots x_{-1},x_0,x_1,\dots]]$ then for $i\in \N$ define $$R_i(f):=f(\dots x_{-i},0,x_{-(i-1)},\dots,x_{-1},x_0,x_1,\dots,x_{i-1},0,x_{i},\dots)$$ in other words, $R_i(f)$ sets $x_i=x_{-i}=0$ and shifts $x_{j}\mapsto x_{j-1}$ and $x_{-j}\mapsto x_{-(j-1)}$ for $j\geq i+1$.
    \end{definition}
    
    \begin{proposition}\label{prop:Ri-op}
    Let $f\in \Q[[\dots x_{-1},x_0,x_1,\dots]]$ have bounded degree, then $$f\in SQSym \iff R_i(f)=R_{i+1}(f)\quad \forall i\in \N.$$
    \end{proposition}
    
    \begin{proof}
    Let us first suppose that $f\in SQSym$. 
    We consider an infinite matrix $c=$ \small ($\begin{smallmatrix}
    c_1 & c_2 & \cdots \\
    c_{-1} & c_{-2} & \cdots
    \end{smallmatrix}$\small) with only a finite number of nonzero entries, and $k\in \N_0$.
    Let us set $\textbf{x}_{k,c}:= x_0^k\prod_{j\in \N}x_j^{c_j}x_{-j}^{c_{-j}}$.
    Then, for any $i\in \N$, the coefficient in $R_{i+1}(f)-R_{i}(f)$ of the monomial $\textbf{x}_{k,c}$ is the difference of $[\textbf{x}_{k,c'}]f$ and $[\textbf{x}_{k,c''}]f$ with $$c'= \begin{pmatrix}
    c_1 & \cdots & c_{i} & 0&  c_{i+1} & \cdots \\
    c_{-1} & \cdots & c_{-i} & 0 &  c_{-(i+1)} & \cdots
    \end{pmatrix}\text{ and } c''= \begin{pmatrix}
\    c_1 & \cdots & c_{i-1} &0 & c_i \cdots \\
    c_{-1} & \cdots & c_{-(i-1)} & 0&  c_{-i} & \cdots
    \end{pmatrix}.$$ Since $f\in SQSym$ these coefficients must be the same. Whence $R_i(f)=R_{i+1}(f)$.
    
    Conversely, let us suppose that for any $i$, $(R_{i+1}-R_i)f = 0$. This implies that for any $c$ as above and any $k\in \N_0$, one has $[\textbf{x}_{k,c'}]f-[\textbf{x}_{k,c''}]f =
     0$. Then note that a transposition $\sigma_i:=(i,i+1)\in \mathfrak{S}$ acts on monomials under the quasi-symmetrizing action \ref{eq:hivert} as: 
     $$\sigma_i\ast \textbf{x}_{k,e}:=\begin{cases} \textbf{x}_{k, \sigma_i \cdot e}& \text{ if } e_i=e_{-i}=0\text{ or if } e_{i+1}=e_{-(i+1)}= 0\\
    \textbf{x}_{k,e}& \text{ otherwise}
    \end{cases}$$  
    where $\sigma_i \cdot e$ is obtained by permuting the columns $i$ and $(i+1)$ in $e =$ \small ($\begin{smallmatrix}
    e_1 & e_2 & \cdots \\
    e_{-1} & e_{-2} & \cdots
    \end{smallmatrix}$\small). Then either $\textbf{x}_{k,e} = \sigma_i\ast \textbf{x}_{k,e}$ or $\{\textbf{x}_{k,e}, \sigma_i \ast \textbf{x}_{k,e}\}=\{\textbf{x}_{k,c'}, \textbf{x}_{k,c''}\}$ for some $c=$ \small ($\begin{smallmatrix}
    c_1 & c_2 & \cdots \\
    c_{-1} & c_{-2} & \cdots
    \end{smallmatrix}$\small). Thus we deduce that $(id-\sigma_i)\ast f = 0$ for all $i\in \N$ and as $\{\sigma_i\}_{i\in \N}$ generates $\mathfrak{S}$ we have: $f\in SQSym$.
    \end{proof}

	\begin{corollary}
	$SQSym$ is a graded $\Q$-algebra. 
	\end{corollary}
	\begin{proof}
	If $f,g\in \Q[[\dots,x_{-1},x_0,x_1,\dots]]$ then 
    $$R_i(f\cdot g) = R_i(f)R_i(g) = R_{i+1}(f)R_{i+1}(g)= R_{i+1}(f\cdot g)\text{ for all } i\in \N.$$ 
    Thus Proposition~\ref{prop:Ri-op} implies that $SQQym$ is stable under multiplication. 
    Also $SQSym= \bigoplus_{d=0}^\infty SQSym^d$ and clearly $f\in SQSym^{d_1}$, $g\in SQSym^{d_2}$ implies $f\cdot g \in SQSym^{d_1+d_2}$.
	\end{proof}

Let $k\in \N_0$ and $\lambda$ be a bicomposition, define the \textit{monomial signed quasisymmetric function} $M_{k, \lambda}$ by
\begin{equation}
M_{k, \lambda}:=\sum_{i_1<i_2<\cdots <i_r} x_0^k x_{i_1}^{a_1}x_{-i_1}^{b_1}\cdots x_{i_r}^{a_r}x_{-i_r}^{b_r}
\end{equation} where the indices $i_1<i_2<\cdots <i_r$ are natural numbers. These functions are the orbit sums of monomials, and they form a basis of $SQSym$ as a vector space. In fact
\begin{equation}
SQSym^d=span_\Q\{M_{k, \lambda} : ~ k+|\lambda|=d \}.
\end{equation}

\begin{example}\label{ex:the-example-bis}
    The chromatic signed quasisymmetric function of Example \ref{ex:the-example} can be expressed as 
    \begin{align*}
    X_{\sdgraph}(\textbf{x};t) = & t^3[M_{0,(\begin{smallmatrix} 1& 1\\0&0
    \end{smallmatrix})}]+t^2[2 M_{0, (\begin{smallmatrix}
        0&1\\1&0
    \end{smallmatrix})}+M_{0,(\begin{smallmatrix} 1& 1\\0&0\end{smallmatrix})} + M_{1,(\begin{smallmatrix}
        1\\0
    \end{smallmatrix})}] \\
    & + t[2 M_{0,(\begin{smallmatrix} 1 & 0 \\ 0 & 1 \end{smallmatrix})} + M_{0,(\begin{smallmatrix} 0&0 \\1 & 1 \end{smallmatrix})}+M_{1,(\begin{smallmatrix}0\\1\end{smallmatrix})}]+[M_{0,(\begin{smallmatrix}
        0 & 0 \\ 1& 1
    \end{smallmatrix})}].
    \end{align*}
\end{example}

Also, as the elements of the monomial basis are indexed by pairs $(k,\lambda)$, we define $A_d:= \dim(SQSym^d)= \# \{(k, \lambda):k\in \N_0, ~ k+|\lambda|=d\}$. This sequence of dimensions, namely $1, 3, 10,34,116\dots$ is \begin{equation}
    A_d=\sum_{k=0}^d\sum_{\substack{c\models d-k\\c=(c_1,\dots,c_p)}}\prod_{i=1}^p (c_i+1).
\end{equation} 
Its generating function is 
\begin{equation}
\sum_{n\geq 0} A_d t^d = \frac{1-t}{1-4t+2t^2}.
\end{equation}
It corresponds to the accumulated sum of \cite{A003480} in OEIS. 

Let us now pay some attention to how to express the product of monomial signed quasisymmetric functions. First we introduce a useful notation for bicompositions; we will represent $\bicomposition = (\lambda_1,\lambda_2,\dots,\lambda_r)$ where 
    $\lambda_i={\big(\begin{smallmatrix}
a_i \\
b_i 
\end{smallmatrix}\big)}$ 
    are the columns of the bicomposition. We also define $(\lambda_1,\dots,\lambda_{r})\cdot(\lambda'_1,\dots,\lambda'_{r'}):= (\lambda_1,\dots,\lambda_r,\lambda_1',\dots,\lambda'_{r'})$ the \textit{concatenation} of two bicompositions.
    
 The \textit{quasi-shuffle} or \textit{stuffle} $\qshuffle$ of bicompositions is the set defined by the following recursive formula. Let $\varnothing$ be the empty composition and $\mathbf{\lambda}=(\lambda_1,\dots,\lambda_r)$, $\mathbf{\lambda'} = (\lambda'_1,\dots,\lambda'_{r'})$
    \[\begin{cases}
        \varnothing \qshuffle \lambda = \lambda \qshuffle \varnothing = \lambda \\
        \lambda \qshuffle \lambda' = \{\lambda_1 \cdot ((\lambda_2,\dots,\lambda_r)\qshuffle \lambda'),\lambda_1' \cdot ((\lambda \qshuffle (\lambda_2',\dots,\lambda_{r'}')), (\lambda_1+\lambda'_1)\cdot ((\lambda_2,\dots,\lambda_r)\qshuffle (\lambda'_2,\dots,\lambda'_{r'}))\}
    \end{cases}\]
\begin{proposition}
$$M_{k_1,\lambda_1} \cdot M_{k_2, \lambda_2} = \sum_{\lambda\in \lambda_1 \qshuffle \lambda_2}M_{k_1+k_2, \lambda}.$$
\end{proposition}
The proof is omitted, it follows the argument of \cite{Ho00},\cite{ABB04}.
\begin{example}
    $$M_{1, (\begin{smallmatrix}
        1&2\\0&1
    \end{smallmatrix})} \cdot M_{2,(\begin{smallmatrix}
        0\\3
    \end{smallmatrix})} = M_{3, (\begin{smallmatrix}
        1&2&0\\0&1&3
    \end{smallmatrix})} +M_{3, (\begin{smallmatrix}
        1 & 0 &2 \\0 & 3&1
    \end{smallmatrix})} +M_{3, (\begin{smallmatrix}
        1&2\\0&3
    \end{smallmatrix})}+M_{3, (\begin{smallmatrix}
        0&1&2\\3&0&1
    \end{smallmatrix})}+M_{3, (\begin{smallmatrix}
        1&2\\3&1
    \end{smallmatrix})}.$$
\end{example}

\begin{remark}
 We may also note that $SQSym$ is in fact a graded bialgebra, hence a Hopf algebra. We shall not develop this point here, we only mention that the coproduct 
 $\Delta: SQSym \rightarrow SQSym \otimes SQSym$ 
 is given by

$$\Delta(M_{k,\lambda}) = \sum_{\lambda_1 \cdot \lambda_2 = \lambda}\sum_{k_1+k_2=k} M_{k_1,\lambda_1}\otimes M_{k_2, \lambda_2}.$$

\end{remark}

\subsection{Signed \texorpdfstring{$(P,\omega)$}{(P,w)}-partitions}
We are now interested in defining a fundamental family for $SQSym$. 
This will be done in the next subsection. We recall that the fundamental basis of $QSym$ is given by the set $\{\Gamma_{P,\omega}:~P\text{ is a chain}\}$, where $\Gamma_{P,\omega}$ is the $(P,\omega)$-partition enumerator, see \cite{S97} for details. In light of this, we are interested in exploring the enumerator of signed $(P,\omega)$-partitions.

The first definition of signed poset was given by Reiner in \cite{R93} in the context of Coxeter groups of type B. Our approach is more similar to that of Gessel or Chow \cite{C01}.

\begin{definition}
\label{def signed permutation}
A permutation $\pi$ of $\Z$ is called a \textit{signed permutation} if $\pi(-i)=-\pi(i)$ for all $i\in \Z$. We denote by $S\mathfrak{S}$ the group of signed permutations with finite support, and by $S\mathfrak{S}_d$ the group of signed permutations of the set $\{0,\pm 1,\dots,\pm d\}$.
\end{definition}

\begin{definition}
A \textit{signed poset} is a poset $P$ with ground set $\pm [d]\backslash \{0\} = \{\pm1,\pm2,\dots,\pm d\}$ such that if $i<_Pj$ then $-j<_P-i$. A \textit{linear extension} $\pi$ of a signed poset $P$ is an element of $S\mathfrak{S}_d$ extending $P$ to a total order, i.e. $i<_Pj$ implies $\pi\inv (i) < \pi \inv (j)$. We will denote by $\mathcal{L}(P)$ the set of linear extensions of $P$.
\end{definition}

We represent $\pi \in \mathcal{L}(P)$ as $\pi=\pi(1)\pi (2)\dots\pi (d)$ and we will note $sgn(\pi)$ the vector of signs $sgn(\pi(1)),sgn(\pi(2)),\dots,sgn(\pi(d))$.

\begin{definition}
    We define the \textit{descent set} of a signed permutation $\pi\in S\mathfrak{S}_d$ as:
    $$DES(\pi):= \{i\in\{0,\dots,d-1\}:~\pi(i)>\pi(i+1)\}$$
    where we note that by definition $\pi(0)=0$.
\end{definition}

\begin{figure}
\includegraphics[width=120mm]{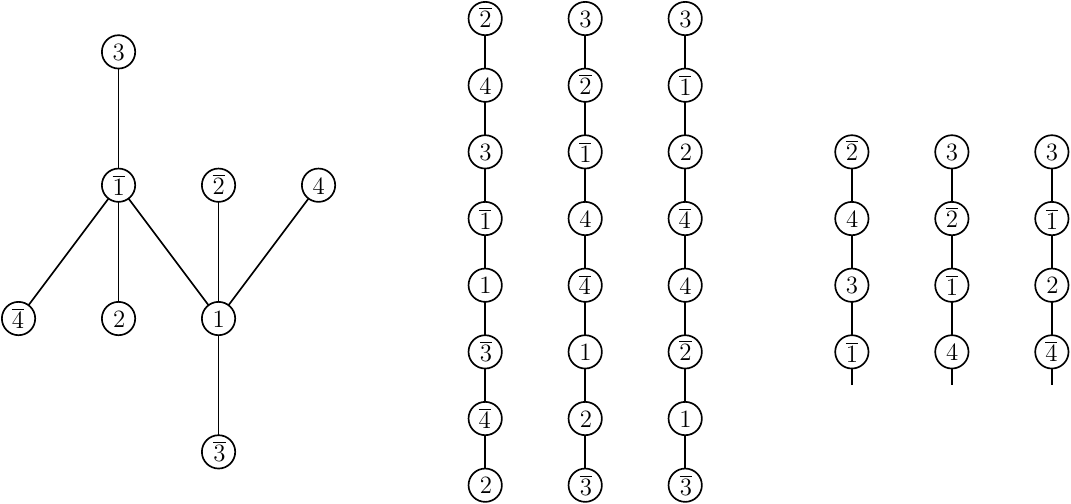}
\caption{The Hasse diagram of a signed poset and three of its linear extensions, that we will note as $\overline{1}34\overline{2}$, $4\overline{12}3$ and $\overline{4}2\overline{1}3$ respectively. When drawing a chain, it is sufficient to specify half of it.}
\end{figure}

\begin{definition}
Let $P$ be a signed poset with ground set $\pm [d]\backslash \{0\}$. A \textit{labeling} of $P$ is a bijection $\omega: \pm [d]\backslash \{0\} \rightarrow \pm[d]\backslash \{0\}$ verifying $\omega(i)<\omega(j)\implies \omega(-j)<\omega(-i)$. A \textit{$(P,\omega)$-partition} is a map $f: P \rightarrow
\Z$ satisfying the conditions:

\begin{itemize}
\item If $x \leq_P y$, then $f(x)\leq f(y)$ ($f$ is order-preserving).
\item If $x\leq_P y$ and $\omega(x)>\omega(y)$, then $f(x)<f(y)$. 
\item $f(x)=-f(-x)$ for all $x\in P$. 
\end{itemize}
We note by $\mathcal{A}(P,\omega)$ the set of $(P,\omega)$-partitions.
\end{definition}

As in the usual case \cite{S97}, a $(P,\omega)$-partition is a map that preserves the order on $P$, where the weakness or strictness of the inequalities is determined by $\omega$.
If $\omega$ is \textit{natural} i.e. $x<y~\implies~\omega(x)<\omega(y)$, then a $(P,\omega)$-partition is just an order-preserving map $f: P \rightarrow \Z$ verifying $f(x)=-f(-x)$ for all $x \in P$. We then call $f$ simply a $P$-partition. Similarly, if $\omega$ is \textit{dual natural} i.e. $x<y~\implies~\omega(x)>\omega(y)$, then a $(P,\omega)$-partition is a strict order-preserving map $f: P \rightarrow \Z$ verifying $f(x)=-f(-x)$ for all $x \in P$, we then call $f$ a strict $P$-partition.

It is clear that we can have several $\omega$'s on a poset $P$ leading us to the same $\mathcal{A}(P, \omega)$. If we have a covering relation $x \lessdot y$ we will draw a double line in the Hasse diagram if $\omega(x)>\omega(y)$ and a simple line if $\omega(x)<\omega(y)$ and then represent a $(P,\omega)$-partition by a Hasse diagram in which we have simple or double lines to represent weak or strict inequalities, see Figure \ref{figure: omega-hasse} as example.

\begin{figure}[ht]

\includegraphics[width=110mm]{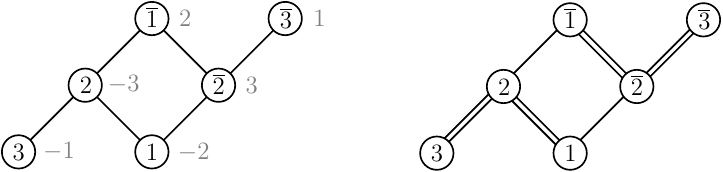}
\caption{
\label{figure: omega-hasse}
The inequalities induced by the $\omega$ on the left are illustrated in the diagram on the right.}
\end{figure}

As a direct consequence from Stanley's result \cite[Lemma 6.1]{S72}  on usual $(P,\omega)$-partitions we have the following result.

\begin{proposition}
Let $\mathcal{L}(P,\omega)$ be the set of linear extensions of 
$P$, where the labeling is inherited from the labeling $\omega$ of $P$. Then
\label{prop:stanley-Pp}
    \begin{equation}
\mathcal{A}(P,\omega)=\bigsqcup_{\pi \in \mathcal{L}(P,\omega)} \mathcal{A}(\pi, \omega)
\end{equation} where $\sqcup$ denotes disjoint union.
\end{proposition}

\begin{definition}
Let $P$ be a signed poset and $\omega$ a signed labeling of $P$. We define the \textit{$(P,\omega)$-partition enumerator} to be the function 
\begin{equation}
\Gamma_{P, \omega}=\Gamma_{P,\omega}(\dots,x_{-1},x_0, x_1,\dots)=\sum_{f\in \mathcal{A}(P, \omega)}x^f=\sum_{f \in \mathcal{A}(P, \omega)}x_{f(1)}x_{f(2)}\cdots x_{f(d)}.
\end{equation}
\end{definition}

\begin{remark}
Note that the definition does not depend on the absolute value of the labels on the poset but it depends on where are the negative labels placed. In this sense, we can represent the $(P,\omega)$-partition enumerator of Figure \ref{figure: omega-hasse} as follows.
\begin{figure}[ht]
\includegraphics[width=40mm]{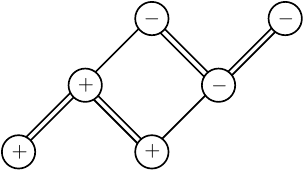}
\end{figure}
\end{remark}

Proposition \ref{prop:stanley-Pp} translates as follows on $\Gamma$ enumerators:
\begin{equation}\label{eq: relation-enumerators}
\Gamma_{P, \omega}=\sum_{\pi \in \mathcal{L}(P,\omega)}\Gamma_{\pi, \omega}.
\end{equation}

\begin{remark}
We observe that:
$\Gamma_{P,\omega}\in SQSym$.
\end{remark}

\subsection{The fundamental family of \texorpdfstring{$SQSym$}{SQSym}}
The aim of this subsection is to present a fundamental generating family $\F$ for the algebra $SQSym$. As we shall see, $\F$ is not a basis. 
We will show how to extract a basis $\bar\F$ from $\F$,
but it is relevant in our context to deal with $\F$ rather than with $\bar\F$ for several reason: $\F$ is canonical since it is defined as the enumerators of the signed labeled chains; $\F$ is more suited to deal with products since it allows a positive expansion for the product of its elements; and we shall see in the next section that it is the appropriate framework to compute an expansion for the chromatic signed quasisymmetric invariant.

\begin{definition}
    As in the classical case for $QSym$, we define the {\em fundamental family} $\F$ of $SQSym$ as the set of enumerators of chains: $$\mathcal{F}:= \{\Gamma_{P,\omega}:~ P \text{ is a signed chain, } \omega \text{ is a signed labeling of }P\}.$$ 
\end{definition}

As said before, there exist linear relations between the elements of $\F$.
\begin{proposition}
\label{prop: the relation}
There is the following relation between enumerators of $(P,\omega)$-partitions:
\begin{figure}[h]
\includegraphics[width=55mm]{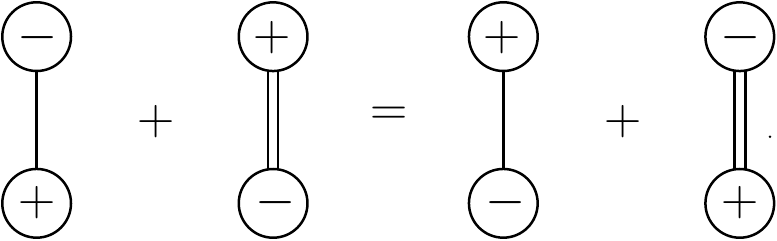}
\end{figure}
\end{proposition}
\begin{proof}
Consider the following evident relation between enumerators of $(P,\omega)$-partitions where $\omega$ is indicated on the left of each element and apply \ref{eq: relation-enumerators} on both sides of the equality.

\begin{figure}[h]
\includegraphics[width = 80mm]{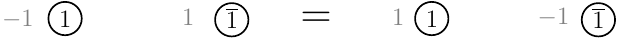}
\end{figure}
\end{proof}

\begin{remark}\label{rem:F-rel}
We notice that the relation in Proposition \ref{prop: the relation} holds as a "pattern relation": it is true for any common chain above and below the two vertices appearing in this relation. See Figure \ref{fig: relation-extended} as an example. 
\begin{figure}
\caption{\label{fig: relation-extended}}
\includegraphics[width = 80mm]{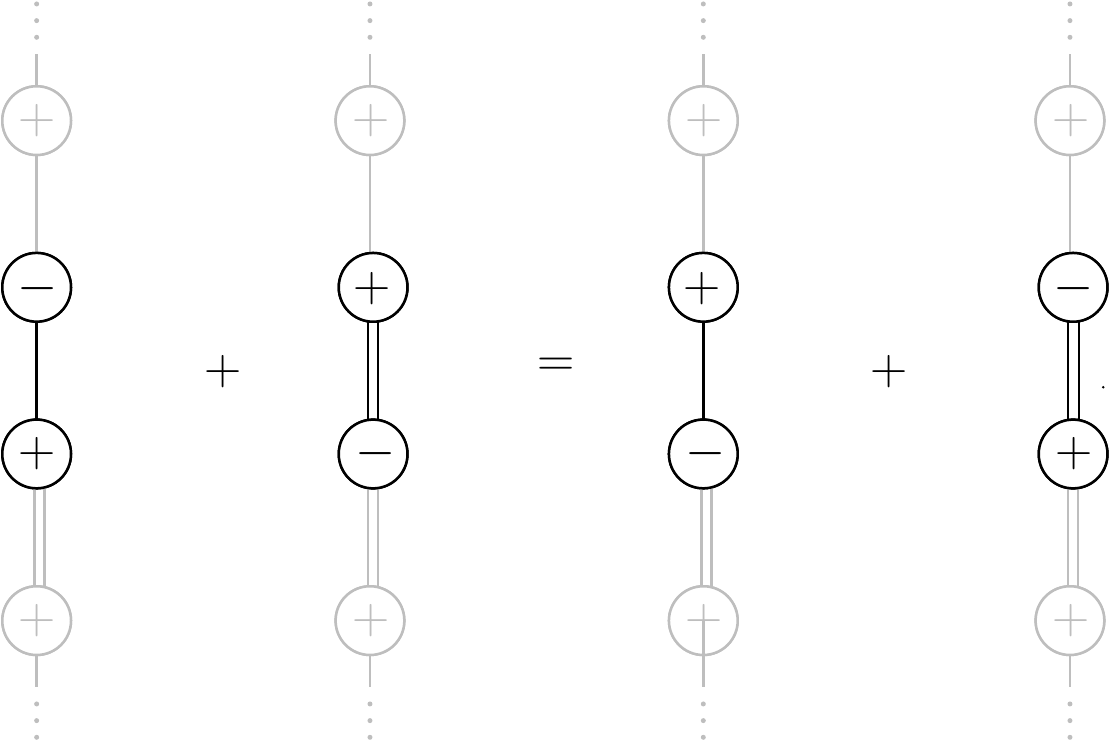}
\end{figure}
\end{remark}

\begin{remark}
 The family $\mathcal{F}$ appears in \cite{W97} in a quite different context: not as $(P,\omega)$-partitions of signed posets, but instead to establish an analogue of Stanley's celebrated result for counting the acyclic orientations of a given graph with a fixed number of sinks.
\end{remark}

Now, let us show how to extract a basis from $\mathcal{F}$. 

We introduce some notation to refer to the elements of $\mathcal{F}$. Let us fix $d\in \N_0$, $S = \{s_1<s_2<\cdots\}\subset\{0,...,d-1\}$ and $\epsilon = (\epsilon_1,...,\epsilon_d)\in \{+,-\}^d$. We define $F_{S}^\epsilon$ as the $(P,\omega)$-partition enumerator of the labeled signed chain of length $2d$ whose strict edges correspond to $S$ and whose signs are given by $\epsilon$. More precisely, the signed chain corresponding to $F_S^\epsilon$ is obtained by drawing its upper half (of size $d$) and then, from bottom to top, by assigning the signs of  $\epsilon$ to the elements of the chain and place strict edges at the positions specified by $S$. If $s_1=0$ the edge that joins the two halves of the chain will be strict, if $m\in S$ with $m>0$ the $m$'th edge of the upper half of the chain will be strict.

\begin{example}
Let $d=4$ then $F_{\{1,3\}}^{(+++-)}$, $F_{\{1,2\}}^{(-++-)}$, $F_{\{0,2\}}^{(--++)}$ and $F_{\{0,3\}}^{(-+-+)}$ are respectively the $(P,\omega)$-partition enumerators of the chains in Figure \ref{fig: example-chains}.
\end{example}

\begin{figure}[h]
\includegraphics[width=105mm]{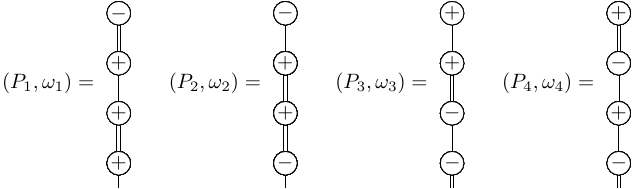}
\caption{\label{fig: example-chains}}
\end{figure}

Thus $F_S^\epsilon$ has the following form 
\begin{equation}
F_S^\epsilon = \sum_{\substack{0\leq i_1\leq \cdots \leq i_d\\i_j<i_{j+1}\text{ if }i\in S\\ 0< i_1 \text{ if } 0\notin S}} x_{\epsilon_1 i_1} x_{\epsilon_2i_2}\cdots x_{\epsilon_3i_3}.
\end{equation}

Now we are interested in extracting a basis $\overline{\mathcal{F}}$ from $\mathcal{F}$. The elements of $\overline{\mathcal{F}}$ will be the $(P,\omega)$-partition enumerators of a particular type of chains, which we will call {\em minimal chains}. 

When regarding a chain we will focus to the weak subchains (that is, subchains whose edges are all weak). 

\begin{definition}
A chain is said to be {\em minimal} if it avoids, 
within a weak subchain, to have a $-$ sign placed above a $+$ sign.

The family $\overline{\mathcal{F}}$ is defined as the subset of ${\mathcal{F}}$ corresponding to minimal chains.
\end{definition}

Note that if we restrict our attention to the upper half of the chain, the condition for a chain to be minimal translates into the following:
\begin{enumerate}
\item any weak subchain must have its signs sorted with the $-$ below and the $+$ above; 
\item if the edge connecting the two halves of the chain is weak, then the first subchain (in the upper half) must have only $+$ signs.
\end{enumerate}

For instance, in Figure \ref{fig: example-chains} $(P_1,\omega_1)$ and $(P_3,\omega_3)$ are minimal, $(P_2,\omega_2)$ and $(P_4,\omega_4)$ are not.

\begin{proposition}
The family $\overline{\mathcal{F}}$ is a basis of $SQSym$.
\end{proposition}

\begin{proof}
We begin by showing that the number of elements of degree $d$ in $\overline{\mathcal{F}}$ is equal to $A_d$, the dimension of $SQSym^d$. To do this, we construct a bijection between minimal chains of length $2d$ and the set $\{(k, \lambda) : k + |\lambda| = d\}$, which indexes the monomial basis in degree $d$. 

Given a minimal chain $(P, \omega)$, we define its corresponding pair $(k, \lambda)$ as follows. If the edge that joins the two halves of the chain is strict then $k = 0$. If it is weak, then $k$ is the number of elements before the next strict edge. Then $\lambda = \bicomposition$ where $a_i$ and $b_i$ are  respectively the number of $+$ and $-$ signs in the $i$'th subchain of weak edges in the upper half of the chain. This assignment can easily be seen to define a bijection between pairs $(k,\lambda)$ and minimal chains.

To show that the elements in $\overline{\mathcal{F}}$ are linearly independent, we consider the lexicographic order induced by the variable order $x_0>x_1>x_{-1}>x_{2} >x_{-2}>\cdots$. It is clear that the leading monomial of the elements of $\overline{\mathcal{F}}$ are all distinct, 
whence the linear independence. 

\end{proof}

\begin{example} For the chains of Figure \ref{fig: example-chains}, the pairs $(k, \lambda)$ corresponding to $(P_1,\omega_1)$ and $(P_3,\omega_3)$ are $(1,\big(\begin{smallmatrix}
2 & 0\\
0 & 1
\end{smallmatrix}\big))$ and $(0, \big(\begin{smallmatrix}
0 & 2\\
2 & 0
\end{smallmatrix}\big))$ respectively. 
 \end{example}

\begin{remark}
We may also prove directly that $\overline{\mathcal{F}}$ generates $\mathcal{F}$. To do this, we observe that Remark \ref{rem:F-rel} 
implies that each time a $-$ sign appears directly above a $+$ sign connected by a weak edge, we may express the first chain in Figure \ref{fig: relation-extended} as a linear combination of the other three.  By iterating this process we can express every signed chain as a linear combination of minimal chains.
\end{remark}

As in the classical case for $QSym$, the fundamental family $\F$
gives a simple and elegant way to express the product in $SQSym$.
This is based on the fact that the product of two $(P,\omega)$-partition enumerators is the enumerator of the disjoint union of the two posets. As a consequence, we get a {\em positive} expansion for the product of elements of $\F$.

To explicit this, we need the following notion of shuffle.
Let us consider two pairs of words
$(\pi,\epsilon) = (\pi_1\dots\pi_{d},\epsilon_1\dots\epsilon_{d})$
and
$(\pi',\epsilon') = (\pi'_1\dots\pi'_{d'},\epsilon'_1\dots\epsilon'_{d'})$.
The (simultaneous) shuffle 
$(\pi,\epsilon) \shuffle (\pi',\epsilon')$
is defined recursively by:
\[\begin{cases}
        \varnothing \shuffle (\pi,\epsilon) = (\pi,\epsilon) \shuffle \varnothing = (\pi,\epsilon) \\
        (\pi,\epsilon) \shuffle (\pi',\epsilon') = 
        \{
        (\pi_1,\epsilon_1) \cdot ((\pi_2\dots\pi_{d},\epsilon_2\dots\epsilon_d) \shuffle (\pi',\epsilon'))
        ,\\
        \ \ \ \ \ \ \ \ \ \ \ \ \ \ \ \ \ \ \ \ \ \ \ \ 
        (\pi'_1,\epsilon'_1) \cdot ((\pi,\epsilon) \shuffle (\pi'_2\dots\pi'_{d'},\epsilon'_2,\dots\epsilon'_{d'}))
        \}
\end{cases}\]
where the concatenation $\cdot$ is meant simultaneously in $\pi$ and $\epsilon$. Given a signed permutation $\pi \in S\mathfrak{S}_{d'}$ and $d\geq 0$, $\overline{\pi}^d$ is defined as the word $$\overline{\pi}^d = \begin{cases}
    \overline{\pi}^d_i = \pi_i +d \text{ if } \pi_i>0\\
    \overline{\pi}^d_i = \pi_i - d \text{ if } \pi_i<0
\end{cases} \text{ for } i = 1,\dots,d'.$$ This way, given two signed permutations $\pi\in S\mathfrak{S}_d$ and $\pi'\in S\mathfrak{S}_{d'}$ the words $\pi$ and $\overline{\pi}^d$ do not have letters in common.

\begin{proposition}
Given $d,d'\geq 0$, $\epsilon\in \{\pm\}^d$, $\epsilon'\in  \{\pm\}^{d'}$ the product in the $\F$ family is given by:
$$F_{S}^{\epsilon} \cdot F_{S'}^{\epsilon'} = \sum_{(\tau,\delta) \in (\pi,\epsilon) \shuffle (\overline{\pi}'^{d},\epsilon')} F_{DES(\tau)}^{\delta},$$
where  $\pi \in S\mathfrak{S}_d$ and $\pi'\in S\mathfrak{S}_{d'}$ are chosen such that $DES(\pi)=S$,
$DES(\pi')=S'$.

\end{proposition}

\begin{proof}
It is a direct consequence of the proposition \ref{prop:stanley-Pp}.
\end{proof}

\begin{example}
As an example, we compute the product of $F_{\{1\}}^{+-}$ and $F_{\{0\}}^{++}$. We choose $\pi = 21$ and $\pi' = \overline{1}2$.

\begin{align*}
    (\pi,-+)\shuffle (\overline{\pi '}^2, ++)  = & (21,-+)\shuffle (\overline{3}4,++) 
    = \{(21\overline{3}4,-+++),(2\overline{3}41,-+++),\\ &(\overline{3}421,++-+), (\overline{3}241,+-++), (\overline{3}214,+-++),(2\overline{3}14,-+++)\}
\end{align*}

And thus we obtain $$F_{\{1\}}^{+-} \cdot F_{\{0\}}^{++} = F_{\{1\}}^{-+++}+F_{\{1,3\}}^{-+++}+F_{\{0,2,3\}}^{++-+}+F_{\{0,3\}}^{+-++} + F_{\{0,2\}}^{+-++} + F_{\{1\}}^{-+++}.$$
\end{example}

\begin{remark}
    As the family $\F$ is not a basis, the expansion obtained in the proposition is not unique and it depends on the choice of $\pi$ and $\pi'$.
\end{remark}

\section{An expansion in the fundamental family}
\label{sec:sec4}

The aim of this section is to express $X_{\sdgraph}(\textbf{x};t)$ in the fundamental generating family $\mathcal{F}$. 

If $\sdgraph = (\Sigma, \tau)$ is an acyclic directed signed graph on $d$ vertices, we can associate to it a signed poset $P_\tau$ on $2d$ vertices. A similar procedure is described in \cite{ADM22} for the unsigned case. For this, consider the directed symmetric graph $G_{\sdgraph}$ with vertex set $\{u_1,u_{-1},\dots,u_d,u_{-d}\}$, then consider in $\{\pm1,\dots,\pm d\}$ the order relation $i\preceq j$ if there is a directed  path from $u_j$ to $u_i$. Clearly $i \preceq j$ if and only if $-j\preceq -i$. In other words, draw $G_{\sdgraph}$ so that all the edges point upwards on the page; removing the arrows results in a Hasse diagram of a signed poset, possibly with some redundant edges. We will note $P_\tau$ the resulting poset, see Figure \ref{fig: correspondance-graphe-poset}.

\begin{figure}
    \centering
    \includegraphics[width=0.8\linewidth]{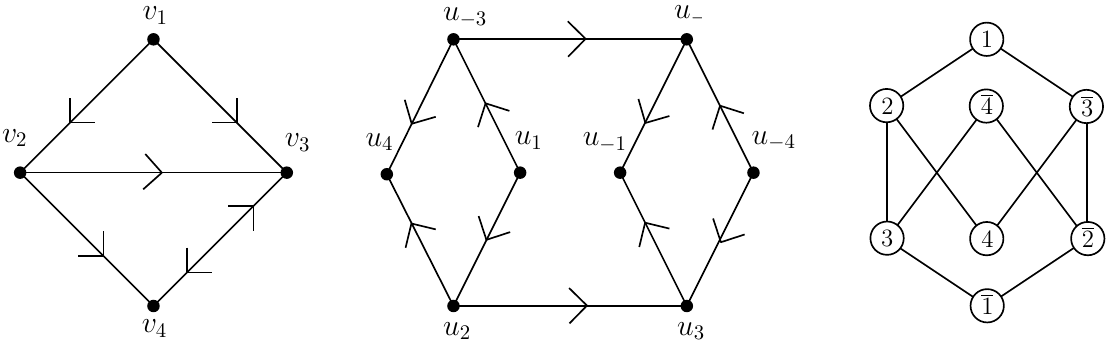}
    \caption[Example-of-a-signed-graph]{\label{fig: correspondance-graphe-poset}Example of a signed graph $\protect\overrightarrow{\Sigma}$, its corresponding symmetric graph $\protect G_{\overrightarrow{\Sigma}}$ and the Hasse diagram of the associated poset $\protect P_\tau$.}
\end{figure}

\begin{definition}
    Let $\pi\in S\mathfrak{S}_d$ and $\sdgraph = (\Sigma, \tau)$ a directed  signed graph on $d$ vertices.
    \begin{itemize}
        \item The {\em $(\Sigma, \pi)$-rank} of $m\in \{\pm1,\dots,\pm d\}$ is the length $\ell$ of the longest subword $\pi_{i_1}\pi_{i_2}\cdots \pi_{i_\ell}$ of $\pi_{-d}\cdots\pi_{-1}\pi_1\cdots\pi_d:= \pi(-d)\cdots \pi(-1)\pi(1)\cdots \pi(d)$ such that $\pi_{i_\ell} = m$ and $\{u_{\pi_{i_j}},u_{\pi_{i_{j+1}}}\}$ is an edge of the symmetric graph $G_\Sigma$ for all $j = 1,2,\dots,\ell-1$. We say that $\pi$ has a $\Sigma$-descent at $i$ with $i\in \{0,1,\dots,d-1\}$ if either of the following conditions holds:
        \begin{itemize}[label = $\ast$]
            \item $rank_{(\Sigma,\pi)}(\pi_i)<rank_{\Sigma, \pi}(\pi_{i+1})$
            \item $rank_{(\Sigma,\pi)}(\pi_i)=rank_{\Sigma, \pi}(\pi_{i+1})$ and $\pi(i)<\pi(i+1)$
        \end{itemize}
        We note $DES_\Sigma(\pi)$ the set of $(\Sigma,\pi)$-descents.
    \item A {\em $\sdgraph$-inversion} of $\pi$ is a directed  edge $(u_i,u_j)$ of $G_{\sdgraph}$ such that $\pi\inv(i)<\pi\inv(j)$, i.e. $i$ precedes $j$ in $\pi$. We define $inv_{\sdgraph}(\pi)$ as half the number of $\sdgraph$-inversions of $\pi$ (if $(u_i,u_j)$ is a $\sdgraph$-inversion of $\pi$ then $(u_{-j},u_{-i})$ too).
    \end{itemize} 
\end{definition}

Before presenting the main theorem of this section,  we observe that given a signed graph $\Sigma$ with $d$ vertices, every signed permutation $\pi\in S\mathfrak{S}_d$ induces an acyclic orientation $\tau_\pi$ on $\Sigma$. Let $\pi=\pi(-d)\cdots\pi(-1)\pi(1)\cdots \pi(d)$ and consider an edge $e = \{u_i,u_j\}$ of $G_{\Sigma}$, then orient $e$ from $u_i$ to $u_j$ if $\pi\inv(i)>\pi\inv(j)$, that is, if $j$ precedes $i$ in $\pi$. This is a signed orientation since $\pi\inv(-i)= -\pi\inv(i)$ and it is clearly acyclic.

\begin{theorem} Given a directed signed graph $\sdgraph = (\Sigma,\tau)$ on $d$ vertices then $$X_{\sdgraph}(\textbf{\em{x}};t) = \sum_{\pi\in S\mathfrak{S}_d}t^{inv_{\sdgraph}(\pi)}F^{sgn(\pi)}_{DES_{\Sigma}(\pi)}.$$
\end{theorem}

\begin{proof}
The proof is greatly inspired from that of Theorem 3.1 in \cite{E17}. Note $AO(\Sigma)$ the set of acyclic orientations of a signed graph $\Sigma$ and $C(\tau)$ the set of proper colorings of $\Sigma$ that are compatible with the orientation $\tau$, then \eqref{equation: suma-signed-chambers} can be rewritten as $$X_{\sdgraph}(\textbf{x};t) = \sum_{\tau' \in AO(\Sigma)}t^{asc_{\sdgraph}(\tau')}\sum_{\kappa\in C(\tau')}\textbf{x}^\kappa$$ where $asc_{\sdgraph}(\tau')$ is the number of edges of $\Sigma$ in which $\tau$ and $\tau'$ have opposite directions; that is, $asc_{\sdgraph}(\tau') = asc(C)$, with  $C$ the chamber induced by $\tau'$. Now, note that $\kappa \in C(\tau')$ if and only if $\kappa$ is a strict $P_{\tau'}$-partition, i.e. if $\omega_{\tau'}$ is an order-reversing labeling of $P_{\tau'}$ then $\kappa\in C(\tau')\iff \kappa$ is a $(P_{\tau'},\omega_{\tau'})$-partition. Then $$\sum_{\kappa\in C(\tau')}\textbf{x}^\kappa = \Gamma_{(P_{\tau'},\omega_{\tau'})} = \sum_{\pi\in \mathscr{L}(P_{\tau'},\omega_{\tau'})}\Gamma_{\pi, \omega_{\tau'}} = \sum_{\pi\in \mathscr{L}(P_{\tau'},\omega_{\tau'})}F_{DES(\pi)}^{sgn(\omega\inv_{\tau'}\pi)}.$$

Let $e: P_{\tau'}\rightarrow \{\pm1,\dots\pm d\}$ be the identity map, i.e. the map that sends each element of $P_{\tau'}$ to its original label. Then $\mathcal{L}(P_{\tau'},e)$ is the set of linear extension of $P_{\tau'}$ with its original labeling $e$. For $\pi\in \mathcal{L}(P_{\tau'},e)$, let $\omega_{\tau'}\pi$ denote the product of $\omega_{\tau'}$ and $\pi$ in $S\mathfrak{S}_d$, then we have $\omega_{\tau'}\pi\in\mathcal{L}(P_{\tau'},\omega_{\tau'})\iff \pi\in\mathcal{L}(P_{\tau'},e)$. Also, as every signed permutation $\pi\in S\mathfrak{S}_d$ induces an acyclic orientation $\tau_\pi$ on $\Sigma$ and there is a bijection between $S\mathfrak{S}_d$ and the linear extensions of acyclic orientations of $\Sigma$ we have that
$$X_{\sdgraph}(\textbf{x};t) = \sum_{\tau' \in AO(\Sigma)}t^{asc_{\sdgraph}(\tau')}\sum_{\pi\in \mathcal{L}(P_{\tau'},e)}F_{DES(\omega_{\tau'}\pi)}^{sgn(\pi)} = \sum_{\pi\in S\mathfrak{S}_d}t^{asc_{\sdgraph}(\tau_\pi)}F_{DES(\omega_{\tau_\pi}\pi)}^{sgn(\pi)}.$$

 We have taken $\omega_{\tau_\pi}$ to be any order reversing map for $P_{\tau_\pi}$, we will now chose a specific one. Let $rank_\tau(m)$ be the rank \footnote{Although the poset $P$ is not necessarily ranked, by the rank of an element $m$ we mean the length of the shortest chain from $m$ to a minimal element of $P$.} of the element $m$ in the poset $P_\tau$, we assign a label from $\{\pm 1, \dots,\pm d\}$ to the elements of $P_{\tau_\pi}$, starting with those of rank 0. If multiple elements have rank 0, we label them according to the original label $e$ from smallest to largest. The labeling begins at $d$ and proceeds in descending order down to $-d$. We then continue with the elements of rank 1, followed by those of rank 2, and so on. This is clearly an order-reversing map.
 
 Note that for all $m\in \{\pm 1, \dots, \pm d\}$, $rank_{(\Sigma, \pi)}(m) = rank_{\tau_\pi}(m)$. Then $m\in DES(\omega_{\tau_\pi})\iff \omega(\pi(m))>\omega(\pi(m+1))$, if and only if one of the following conditions hold:
 \begin{itemize}
     \item $rank_{\tau_\pi}(m)<rank_{\tau_\pi}(m+1)$,
     \item $rank_{\tau_\pi}(m)=rank_{\tau_\pi}(m+1)$ and $\pi(m)<\pi(m+1)$. 
\end{itemize}
So we have $DES(\omega_{\tau_\pi}) = DES_\Sigma(\pi)$. Also, as $asc_{\sdgraph}(\tau_\pi)$ is the number of edges in which $\tau_\pi$ does not match the original orientation $\tau$ on $\Sigma$, it is not difficult to see that $asc_{\sdgraph}(\tau_\pi) = inv_{\sdgraph}(\pi)$, and we get the desired formula.
\end{proof}

\begin{example}\label{ex:F-dev}
    Consider $\sdgraph = (\Sigma, \tau)$ the directed signed graph of Example \ref{ex:the-example}. Then the statistics for each of the $\pi\in S\mathfrak{S}_2$ are listed in Table \ref{table: tabla-ejemplo-deux-sommets}. Thus $X_{\sdgraph}(\textbf{x};t)$ is expressed as
    \begin{align*}
        X_{\sdgraph}(\textbf{x};t) = & t^3 F_{\{0,1\}}^{++} + t^2 [F_{\{0,1\}}^{++} + F_{\{0,1\}}^{-+} +  F_{\{1\}}^{-+}] + t [2 F_{\{0,1\}}^{+-}+ F_{\{1\}}^{--}]+ F_{\{0,1\}}^{--}.
    \end{align*}
This coincides with the formula given in Example \ref{ex:the-example-bis}
\end{example}

\begin{table}[h]
\caption{
\label{table: tabla-ejemplo-deux-sommets}
Table for Example~\ref{ex:F-dev}
}
\begin{tabular}{|c|c|c|c|c|}
\hline
$\pi$ & $inv_{\overrightarrow{\Sigma}(\pi)}$ & $sgn(\pi)$ & $(\Sigma,\pi)-rank$ & $DES_\Sigma(\pi)$ \\ \hline
$\overline{21}12$ & 3 & ++ & 1234 & \{0,1\} \\
$\overline{12}21$ & 2 & ++ & 1223 & \{0,1\} \\
$2\overline{1}1\overline{2}$ & 1 & $+-$ & 1234 & \{0,1\} \\
$2\overline{1}1\overline{2}$ & 2 & $-+$ & 1223 & \{1\} \\
$\overline{2}1\overline{1}2$ & 2 & $-+$ & 1234 & \{0,1\} \\
$1\overline{2}2\overline{1}$ & 1 & $+-$ & 1223 & \{0,1\} \\
$21\overline{12}$ & 0 & $--$ & 1234 & \{0,1\} \\
$12\overline{21}$ & 1 & $--$ & 1223 & \{0\} \\ \hline
\end{tabular}
\end{table}

\section{Symmetry}
\label{sec:sec5}

In this section, we study the (signed) symmetry of $X_{\sdgraph}(\textbf{x};t)$. That is, we want to know for which directed signed graphs $\sdgraph$, one has $X_{\sdgraph}(\textbf{x};t)\in SSym[t]$. In other words, given $\sdgraph = (V, E, \sigma, \tau)$ we want to know whether the coefficient of $t^a$, for $a\in\{0,\dots,|E|\}$, is invariant under the action of $S\mathfrak{S}_{|V|}$ on its variables $\textbf{x} = (\dots,x_{-1},x_0,x_1,\dots)$.

The similar question for (unsigned) directed graphs has been investigated in \cite{E17}, where it is proven that the chromatic quasisymmetric function is symmetric for a class of directed graphs named proper circular arc digraphs. In \cite{GPS24} they study the same question for $X_{G}(x;t)$ where $G$ is a labeled graph and $X_{G}(x;t)$ is the invariant introduced in \cite{SW16}. This corresponds to restricting to the case of {\em acyclic} directed graphs. We will see in Proposition \ref{prop:sym-no1} that this case is not interesting in the context of signed graphs. 

We show that for directed signed graphs, the symmetry implies much stronger constraints. On the one hand, Theorem \ref{theo:ssym} exhibits a class of directed signed graphs for which we prove the symmetry. But on the other hand, Proposition \ref{prop:sym-no1}
and Remark \ref{rem:sym-no2} show that we cannot expect a much wider framework where symmetry holds.

One first observation is that if $\dgraph = (V,E)$ is a directed graph such that its chromatic quasisymmetric function is symmetric, this is $X_{\dgraph}(x;t)\in Sym[t]$, then consider $\sdgraph = (V,E,\sigma, \tau)$ the directed signed graph obtained by considering all edges of $\dgraph$ as positive edges and $\tau$ being the original orientation in $\dgraph$ then $X_{\sdgraph}(\textbf{x};t)\in SSym[t]$. To see this it suffices to consider $X_{\dgraph}(x;t)$ in the set of variables $\textbf{x} = (\dots,x_{-1},x_0,x_1,\dots)$ and the condition of invariance under all permutations implies the invariance under signed permutations.

\begin{proposition}
\label{prop:sym-no1}
Let $\sdgraph = (V,E,\sigma,\tau)$ be a directed signed graph having $\sigma(e) = -$ for some $e\in E$ and $\tau$ be an acyclic orientation, then $X_{\sdgraph}(\textbf{\em{x}};t)\notin SSym[t]$.
\end{proposition}

\begin{proof}
Say $e = \{i,j\}$, then $\{x_i = -x_j\}\in \mathscr{H}_\Sigma$, and this hyperplane separates the points having all positive coordinates from those having all negative coordinates. As the orientation in $\sdgraph$ is acyclic there exists a chamber $C$ corresponding to this orientation.
This chamber corresponds to the constant term (in $t$) of 
$X_{\sdgraph}(\textbf{x};t)$. In this chamber we can find an integer point $\alpha \in \Z^d$ whose coordinates are all different in absolute value leading to the monomial $x_{\alpha_1}\cdots x_{\alpha_d}$. But, as we cannot find two integer points in $C$ leading to the monomials $x_{|\alpha_1|}\cdots x_{|\alpha_d|}$ and $x_{-|\alpha_1|}\cdots x_{-|\alpha_d|}$ respectively, the constant term of $X_{\sdgraph}(\textbf{x};t)$ is not signed symmetric. Thus $X_{\sdgraph}(\textbf{x};t)\notin SSym[t]$.
\end{proof}

Now we present a family of directed signed graphs, which are a particular case of the \textit{circular indifference directed  graphs} defined in \cite{E17}.

\begin{definition}
    Let $k<d-1\in \N$ then define $\overrightarrow{C}_{d,k} = (V,E)$ to be the directed  graph with vertex set $V = \{v_1,\dots,v_d\}$ and (directed) edge set $$E = \{(v_i,v_{i+p (mod~d)}): ~ i\in\{1,\dots,k\} ~ p\in \{1,\dots,k\}\}.$$
\end{definition}

These graphs are indifference graphs which are invariant under rotation. We will see in Remark \ref{rem:sym-no2} why we impose this condition.

\begin{figure}
    \centering
    \includegraphics[width=0.5\linewidth]{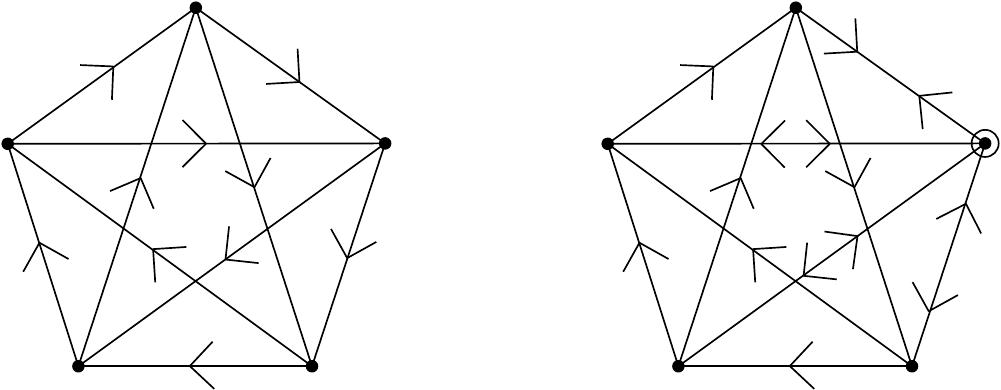}
    \caption{\label{fig: example-pentagon} $\protect \overrightarrow{\Sigma}_{5,2}$ and $\protect \overrightarrow{\Sigma}_{5,2}^\nu$.}
\end{figure}

\begin{definition}
We consider the graph $\sdgraph_{d,k} = (V,E,\sigma, \tau)$, defined as the graph $\overrightarrow{C}_{d,k} = (V,E)$ in which we consider every edge as a positive edge and $\tau$ the same orientation. Let $\nu: V\rightarrow \{+,-\}$ be a switching function sending all vertices to $+$ except one. 

Figure \ref{fig: example-pentagon} shows an example.
\end{definition}

\begin{theorem}
\label{theo:ssym}
     With these notations, $X_{\sdgraph_{d,k}^\nu}(\mathbf{x};t)\in SSym[t].$
\end{theorem}
\begin{proof}
    We need to prove that the coefficient of $t^a$ for $a\in \{0,\dots,|E|\}$ is invariant under the action of $S\mathfrak{S}$ on its variables. Recall that $$[t^a]X_{\sdgraph} = \sum\limits_{\substack{\kappa \text{ proper  } \\ asc(\kappa) = a}}\textbf{x}^\kappa.$$ We want to show that given any $\pi\in S\mathfrak{S}$, then $$\sum\limits_{\substack{\kappa \text{ proper  } \\ asc(\kappa) = a}}x_{\kappa(v_1)}\cdots x_{\kappa(v_d)} = \sum\limits_{\substack{\kappa \text{ proper  } \\ asc(\kappa) = a}}x_{\pi(\kappa(v_1))}\cdots x_{\pi(\kappa(v_d))}.$$ For each $\pi\in S\mathfrak{S}$, we will  define an involution $\Phi_\pi$ on $\mathcal{C}(\sdgraph_{d,k}^\nu)$, the set of proper colorings of $\sdgraph_{d,k}^\nu$, such that:
    \begin{enumerate}
        \item $asc(\kappa) = asc(\Phi_\pi(\kappa))$
        \item $\#(\Phi_\pi(\kappa))\inv (m) = \# (\pi \circ \kappa)\inv(m)$
    \end{enumerate}
    for each proper coloring $\kappa$ of $\sdgraph_{d,k}^\nu$ and for each $m\in \Z$. 
    
    Note that $\sigper$ is generated by $\pi_i:=(i,i+1)(\overline{i},\overline{i+1})$ for $i\in \N$ and $\pi_0:=(1,\overline{1})$. Thus it suffices to define $\Phi_i:=\Phi_{\pi_i}$ for $i\in \N_0$. Recall that the proper colorings of $\sdgraph _{d,k}$ and the ones of $\sdgraph _{d,k}^\nu$ are in bijection and that this bijection preserves the number of ascents. We will define an involution $\tilde{\Phi}_i$ between proper colorings of $\sdgraph_{d,k}$, and then $\Phi_i$ will be the image map under $\nu$ between proper colorings of $\sdgraph_{d,k}^\nu$.
    
    \[\begin{tikzcd}
	{\mathcal{C}(\overrightarrow{\Sigma}_{d,k}^\nu)} && {\mathcal{C}(\overrightarrow{\Sigma}_{d,k}^\nu)} \\
	{\mathcal{C}(\overrightarrow{\Sigma}_{d,k})} && {\mathcal{C}(\overrightarrow{\Sigma}_{d,k})}
	\arrow["{\Phi_i}", from=1-1, to=1-3]
	\arrow["{\nu}", leftrightarrow, from=1-1, to=2-1]
	\arrow["{\nu}", leftrightarrow, from=1-3, to=2-3]
	\arrow["{\tilde{\Phi}_i}"', from=2-1, to=2-3]
\end{tikzcd}\]

     Let us first suppose $i\in\N$, and define $\Phi_i$. Given $\kappa$ a proper coloring of $\sdgraph_{d,k}$ we consider $\sdgraph_i$ the induced directed  subgraph containing only the vertices colored by $i$ and $i+1$. As $\sdgraph_{d,k}$ is a circular indifference graph, by Lemma 5.4 of [E17 +], the connected components of this graph are either directed  paths or directed  cycles of even length. We will then define a coloring obtained from $\kappa$, distinguishing cases on each connected component of $\sdgraph_i$.
    \begin{itemize}
        \item If it is a cycle or a directed  path with an even number of vertices not containing $v'$, the vertex verifying $\sigma(v') = -$, then we leave $\kappa$ unchanged.

        \item If it is a directed  path with an odd number of vertices or a cycle containing $v'$, we exchange in this component the occurrences of $i$ and $i+1$. The colors not belonging to the component remain unchanged, 

        \item If it is a directed  path $v_{i_1},v_{i_2},\dots,v_{i_\ell}$ with an even number of vertices containing $v'$, then we pair adjacent vertices as $(v_{i_1},v_{i_2}),\dots,(v_{i_{\ell -1}},v_{\ell})$, and if $\hat{v}$ is the vertex in the same pair as $v'$ consider in $\sdgraph_{d,k}$ the rotation $\rho$ that sends $\hat{v}$ to $v'$ and define the coloring sending $v$ to $\kappa(\rho\inv (v))$ for all $v$ in $\sdgraph_{d,k}$.
        {\em This point justifies the condition of invariance under rotation imposed to $\sdgraph_{d,k}$.}

        \begin{figure}[h]
            \centering
            \includegraphics[width=0.65\linewidth]{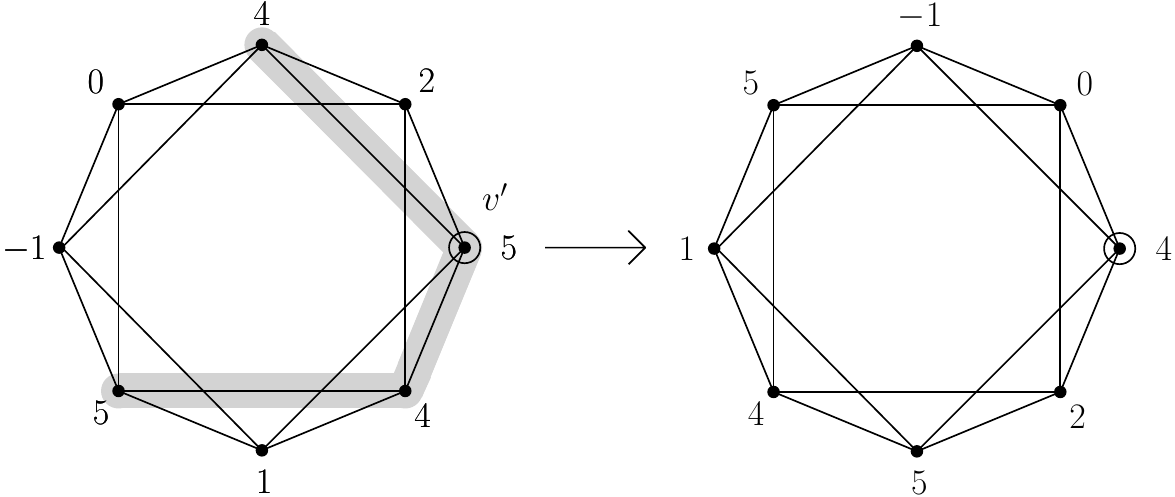}
            \caption{\label{fig: directed-path}
            Example of a directed  path with an even number of vertices for $i=4$.}
            
        \end{figure}
    \end{itemize}
    We may deal exactly in the same way $\sdgraph_{-i}$, the induced directed  subgraph containing only the vertices colored by $-i$ and $-(i+1)$. 
    Thus we associate to $\kappa$ a proper coloring $\tilde{\Phi}_i(\kappa)$. It is to check that $\tilde{\Phi}_i$ is an involution and that the number of ascents is preserved (condition (1)). Also the number of times that the colors $i$ and $i+1$ have been used have been exchanged, the same is true for the colors $-i$ and $-(i+1)$. As these are the only integers affected by $\pi_i$, condition (2) is satisfied.

    It now remains to deal with $\pi_0 = (1,-1)$. Note that we have described a method to get a coloring that exchanges the number of occurrences of two consecutive integers and that keeps the number of ascents invariant. Also note that we can decompose $(1,-1) = (0,1)(0,-1)(0,1)$ so by applying this method 3 times we can define $\Phi_0$ satisfying the required conditions.
\end{proof}

\begin{figure}[ht!]
    \centering
    \includegraphics[width=0.85\linewidth]{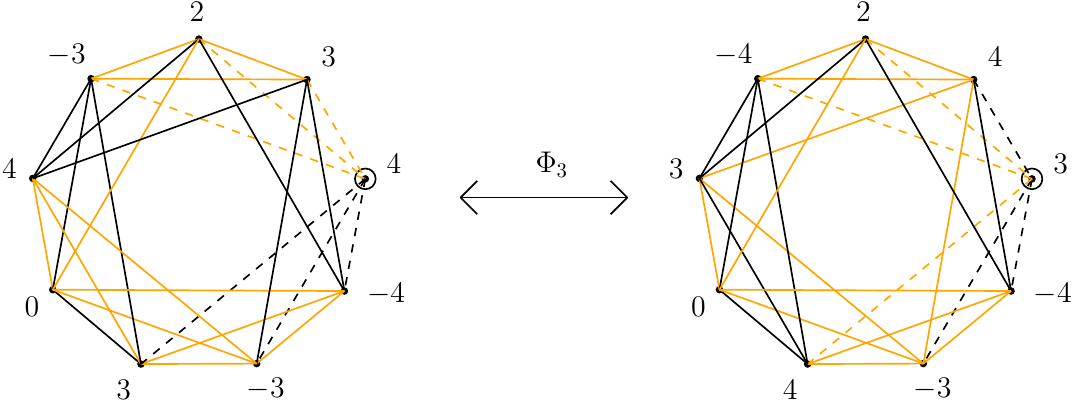}
    \caption{\label{fig: eneagon}
    Example of the involution $\Phi_3$ where the switched vertex is circled and the figure shows the coloring of $\protect \sdgraph_{9,3}$. The edges corresponding to ascents are drawn in orange. The corresponding monomials are $t^{15}x_0x_2x_3^2x_{-3}^2x_4^2x_{-4}$ and $t^{15} x_1x_2x_3^2x_{-3}x_4^2x_{-4}^2$ respectively, which are reciprocally imaged by the transposition $\pi_3$.}
    
\end{figure}

\begin{remark}
\label{rem:sym-no2}
We want to emphasize on the fact that if we relax the conditions in Theorem \ref{theo:ssym}, then the (signed) symmetry is no longer guaranteed. 
We refer to Figure \ref{figure: counterexample2} which shows two directed signed graphs whose chromatic quasisymmetric function is not in $SSym$:
\begin{itemize}
    \item 
in the one on the left, we have relaxed the condition of invariance by rotation;
    \item 
in the one on the right, we have applied a switching whose value is $-1$ on more than one vertex (two). 
\end{itemize}
\end{remark}

\begin{figure}[ht!]
    \centering
    \includegraphics[width=0.5\linewidth]{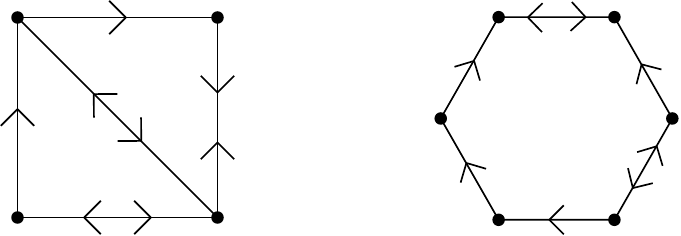}
    \caption{\label{figure: counterexample2}
Two examples of directed signed graphs $\protect\sdgraph$ such that $X_{\protect\sdgraph}(\protect\textbf{x};t)\notin SSym$.}
\end{figure}

\bibliographystyle{amsalpha}

\end{document}